\documentclass[12pt, reqno]{amsart}
\usepackage[parfill]{parskip}    
\usepackage{graphicx}
\usepackage{amssymb}
\usepackage{amsmath,amsthm,amscd}
\usepackage{epstopdf}
\usepackage{amsfonts}
\usepackage{latexsym}

\newtheorem{lemma}{Lemma}[section]
\newtheorem{corollary}[lemma]{Corollary}
\newtheorem{theorem}[lemma]{Theorem}
\newtheorem{proposition}[lemma]{Proposition}
\newtheorem{remark}[lemma]{Remark}
\newtheorem{definition}[lemma]{Definition}

\newtheorem{question}[lemma]{Question}

\input xy
\xyoption{matrix} \xyoption{arrow} \xyoption{curve}
\begin{document}

\title[PI degree parity in $q$-skew polynomial rings]{PI Degree Parity in $q$-Skew Polynomial Rings}%
\author{Heidi Haynal}
\address{Department of Mathematics, University of California, Santa Barbara, California  93106}
\email{heidi@softerhardware.com}

\thanks{This research will form a part of the author's PhD dissertation at the University of California at Santa Barbara.}


\begin{abstract}
For $k$ a field of arbitrary characteristic, and $R$ a $k$-algebra, we show that the PI degree of an iterated skew polynomial ring $R[x_1;\tau_1,\delta_1]\dotsb[x_n;\tau_n,\delta_n]$ agrees with the PI degree of $R[x_1;\tau_1]\dotsb[x_n;\tau_n]$ when each $(\tau_i,\delta_i)$ satisfies a $q_i$-skew relation for $q_i \in k^{\times}$ and extends to a higher $q_i$-skew $\tau_i$-derivation.  We confirm the quantum Gel'fand-Kirillov conjecture for various quantized coordinate rings, and calculate their PI degrees.  We extend these results to completely prime factor algebras. 
\end{abstract}
\maketitle

\section{Introduction}

Presented here is a new technique for analyzing skew polynomial rings satisfying a polynomial identity with an eye toward discovering their PI degrees.  It combines and extends the methods of J\o ndrup \cite{jondrup} and Cauchon \cite{cau}, who introduced techniques of ``deleting derivations" in skew polynomial rings, by means of which they showed that some properties of certain types of iterated skew polynomial ring $A= k[x_1][x_2;\tau_2,\delta_2] \dotsb [x_n;\tau_n,\delta_n]$ are determined by the corresponding ring $A'= k[x_1][x_2;\tau_2] \dotsb [x_n;\tau_n]$.  J\o ndrup's results imply that $A$ and $A'$ have the same PI degree under certain hypotheses, including characteristic zero for the base field.  Cauchon developed an algorithm that gives an isomorphism between certain localizations of $A$ and $A'$, but this requires a $q_i$-skew condition on each $(\tau_i,\delta_i)$ with $q_i$ not a root of unity, which usually precludes $A$ from satisfying a polynomial identity.  We relax the restrictions placed on the base field and its chosen scalars by J\o ndrup and Cauchon, respectively, by introducing the notion of a {\em higher $q$-skew $\tau$-derivation}. 

If we ``twist" the multiplication in the (commutative) coordinate ring of affine, symplectic, or Euclidean $n$-space over a field $k$, we get a (noncommutative) quantized coordinate ring which has the structure of an iterated skew polynomial ring with coefficients in $k$.  This structure is also exhibited in the quantized Weyl algebras and in the quantized coordinate ring of $n \times n$ matrices over $k$.  Letting $A$ represent one of these $k$-algebras, the {\em quantum Gel'fand-Kirillov conjecture} asserts that ${\rm Fract}\, A$ is isomorphic to the quotient division ring of a quantum affine space over a purely transcendental extension of $k$. For more information on the quantum Gel'fand-Kirillov conjecture and proofs of conditions under which the result holds, see \cite{Alev} \cite{cliff} \cite{jordan} \cite{mospan} \cite{panskew} \cite{panov1}.  We will confirm some of these cases in a new way.

The first section sets up the conventions under which we work, including definitions and an established result concerning the PI degree of quantum affine space.  We assume that the reader has some familiarity with the subject, so we do not give an exhaustive collection of definitions.  A comprehensive discussion of any unfamiliar terms can be found in \cite{bluebook} \cite{barcelona} and \cite{brownbook}.  In the second section we define {\em higher $\tau$-derivations} and give necessary and sufficient conditions for their existence.  Of particular interest are higher $\tau$-derivations which satisfy a $q$-skew relation.  In the third section we present a structure theorem for a localization of $q$-skew polynomial rings.  This extends the work of Cauchon \cite{cau}, and the calculations are simplified by the presence of higher $q$-skew $\tau$-derivations.   In the fourth section we deal with the structure of iterated skew polynomial rings. Sometimes it is advantageous to rearrange the order in which the indeterminates appear, so we establish a sufficient condition that allows such reordering.  The main theorem there asserts that if $A$ is an iterated $q$-skew polynomial ring with certain higher $\tau$-derivations, then there is a finitely generated Ore set $T \subseteq A$ such that $AT^{-1}$ is isomorphic to a localization of a much ``nicer" iterated skew polynomial ring.  In the fifth section, we use the tools developed in the previous sections to confirm certain cases of the  quantum Gel'fand-Kirillov conjecture and to find the PI degree of some quantized coordinate rings and quantized Weyl algebras.  In the last section, we follow up with a structure theorem for completely prime factors of iterated skew polynomial rings.  We also present an open question which, if answered positively, would show that the quantum Gel'fand-Kirillov conjecture holds for certain of the prime factor algebras we study.

Throughout, $k$ will denote a field of arbitrary characteristic, $q \in k$ a nonzero element.  The following assumptions apply to all skew polynomial rings that we will consider:
\begin{itemize}
\item all coefficient rings are $k$-algebras
\item all automorphisms are $k$-algebra automorphisms
\item all skew derivations are $k$-linear
\item in all skew  polynomial rings $R[x; \tau, \delta]$, $\tau$ is an automorphism, not just an endomorphism.
\end{itemize}

To say that $R[x;\tau, \delta]$ is a {\em $q$-skew polynomial ring} means that the auomorphism and skew derivation satisfy the relation $\delta\tau=q\tau\delta$.   The reader will note that this is opposite to Cauchon's conventions, but it matches the presentation in \cite{primesprqwa} and others.   To say that $\delta$ is {\em locally nilpotent} means that for every $r\in R$ there is an integer $n_r \ge 0$ such that $\delta^{n_r}(r)=0$, and $\delta^p(r)\ne 0$ for $p<n_r$. Such $n_r$ is called the {\em $\delta$-nilpotence index of $r$}.  The symbol $\mathbb N$ refers to the set of positive integers.  
For a real number $m$ we use the notation $\lfloor m \rfloor$ in section five to indicate the integer part of $m$.

\begin{definition}
{\rm We say that two rings $R$ and $S$ exhibit} PI degree parity {\rm when these two conditions are satisfied:}
\begin{itemize}
\item[]{\rm (1) $R$ is a PI ring if and only if $S$ is a PI ring,}
\item[]{\rm (2) PIdeg$\,R$ = PIdeg$\,S$.}
\end{itemize}
\end{definition} 

For a field $k$ and multiplicatively antisymmetric $\boldsymbol \lambda \in M_n(k)$, the corresponding {\em multiparameter quantum affine space} is the $k$-algebra $\mathcal O_{\boldsymbol \lambda} (k^n)$ with generators $x_1, \dotsc, x_n$ and relations $x_i x_j = \lambda_{ij}x_jx_i$ for all $i,\,j$.  The corresponding {\em multiparameter quantum torus} is the $k$-algebra $\mathcal O_{\boldsymbol \lambda}((k^{\times})^n)$ given by generators $x_1^{\pm 1}, \dotsc, x_n^{\pm 1}$ and the same relations.  The multiplicative set generated by $x_1, \dotsc, x_n$ in $\mathcal O_{\boldsymbol \lambda} (k^n)$ is a denominator set, and  $\mathcal O_{\boldsymbol \lambda}((k^{\times})^n)$ is a localization of  $\mathcal O_{\boldsymbol \lambda}(k^n)$ with respect to this set.

In this paper we'll show that iterated skew polynomial algebras covering a large class of standard examples have PI degree parity with  $\mathcal O_{\boldsymbol \lambda}(k^n)$ for an appropriately chosen $\boldsymbol \lambda$.  To find out what that PI degree may be, we utilize a result of De Concini and Procesi.  In \cite[Proposition 7.1]{DP}, they establish the following formula for calculating the PI degree of a quantum affine space  $\mathcal O_{\boldsymbol \lambda}(k^n)$.  Their assumption of characteristic zero from \cite[Section 4]{DP} is not used in this result. 

\begin{theorem}\label{dp}
{\rm [De Concini - Procesi]} Let ${\boldsymbol {\lambda}} = (\lambda_{ij})$ be a multiplicatively antisymmetric $n \times n$ matrix over $k$.

{\rm (1)} The quantum affine space $\mathcal O_{\boldsymbol{\lambda}}(k^n)$ is a PI ring if and only if all the $\lambda_{ij}$ are roots of unity.  In this case, there exist a primitive root of unity $q \in k^{\times}$ and integers $a_{ij}$ such that $\lambda_{ij} = q^{a_{ij}}$ for all $i, j$.

{\rm (2)}  Suppose $\lambda_{ij} = q^{a_{ij}}$ for all $i, j$, where $q \in k$ is a primitive $\ell^{th}$ root of unity and the $a_{ij} \in \mathbb Z$. Let $h$ be the cardinality of the image of the homomorphism 
\begin{equation*}
\mathbb Z^n \xrightarrow {\quad (a_{ij}) \quad} \mathbb Z^n \xrightarrow {\quad \pi \quad} (\mathbb Z / \ell \mathbb Z)^n
\end{equation*}
where $\pi$ denotes the canonical epimorphism.  Then {\rm PI-deg}\,$(\mathcal O_{\boldsymbol{\lambda}}(k^n)) = \sqrt{h}$.
\end{theorem}

\section{Higher q-Skew $\tau$-Derivations}

Before the featured definition, a brief discussion of a tool used to study $q$-skew polynomial rings is needed.  Having the $q$-skew relation $\delta \tau= q\tau \delta$ in place allows us to group terms of the same degree when we do skew polynomial arithmetic. The means to do this are provided by the $q$-Liebnitz rules.

\begin{definition}\label{tbinomcoeff}
 {\rm For an indeterminate $t$, and integers $n\ge m \ge 0$, we define the following polynomial functions:
\begin{align}
 (m)_t& = t^{m-1} + t^{m-2} + \cdots + t + 1\\
 (m)!_t& =(m)_t (m-1)_t \cdots (1)_t, \text{ and  } (0)!_t = 1\\
  \binom{n}{m}_t&=  \frac{(n)!_t}{(m)!_t (n-m)!_t} 
\end{align}

The expressions $\binom{n}{m}_t$ are called the {\it $t$-binomial coefficients}, or {\em Gaussian polynomials}.  The $t$-binomial coefficients have properties similar to those of the regular binomial coefficients.  Two that will be useful for this work are:
\begin{equation}
\binom{n}{0}_t  = \binom{n}{n}_t = 1 \quad \text{for all  } n \ge 0
\end{equation}
\begin{equation}
\begin{split}
\binom{n}{m} _t  & =  \binom{n-1}{m}_t + t^{n-m} \binom{n-1}{m-1}_t \\
& =  \binom{n-1}{m-1}_t + t^m \binom{n-1}{m}_t \quad  \text{for all   } 0 < m < n
\end{split}
\end{equation}

Proofs for these identities may be found in combinatorics texts such as \cite{stanley}.  When we evaluate the $t$-binomial coefficients at $t=q$, we obtain the $q$-binomial coefficients that we need for studying $q$-skew polynomial rings.

As shown in \cite[Section 6]{primesprqwa}, the following $q$-Liebnitz rules hold for any $q$-skew polynomial ring $R[x;\tau, \delta]$:
\begin{align*}
\delta^n (rs) &= \sum \sb {i=0} \sp n \binom{n}{i} \sb q \tau^{n-i} \delta^i (r)\delta^{n-i} (s)  \text{ for all } r,\,s \in R \text{ and } n=0,1,2,...\\
x^n r &= \sum \sb {i=0} \sp n \binom{n}{i} \sb q \tau^{n-i} \delta^i (r)x^{n-i}  \text{ for all } r \in R \text{ and } n=0,1,2,...
\end{align*} }
\end{definition}

Now, taking a cue from the study of Schmidt differential operator rings, for instance \cite{malm}, we define a sequence of $k$-linear maps that allows us to broaden the class of rings for which we may derive results like those of J\o ndrup and Cauchon.

\begin{definition}\label{highder}
{\rm A {\em higher q-skew $\tau$-derivation} (h.$q$-s.$\tau$-d.) on a $k$-algebra $R$  is a sequence  $d_0, \, d_1, \, d_2,\,\dots$   of $k$-linear operators on $R$ such that
\begin{align*}
&d_0 \text{ is the identity}\\
& d_n(rs) = \sum_{i=0}^n \tau^{n-i}d_i(r)d_{n-i}(s) \, \, \text{for all } r, \, s\in R \text{ and all } n\\
&d_i\tau = q^i\tau d_i \text{ for all } i.
\end{align*} 
If a sequence of $k$-linear maps satisfies the first two conditions, we refer to it as a {\em higher $\tau$-derivation}.  We abbreviate the sequence $\{ d_i \}_{i=0}^\infty$ usually as just $\{ d_i \}$.
A h.$q$-s.$\tau$-d is {\em locally nilpotent} if for all $r \in R$, there exists an integer $n \ge 0$ such that $d_i(r)=0$ for all $i \ge n$, and $d_p(r) \ne 0$ for $p < n$.  In this case, $n$ is called the {\em d-nilpotence index of r}.  A h.$q$-s.$\tau$-d is {\em iterative} if $d_i d_j = \binom{i+j}{j}_q d_{i+j}$ for all $i, \, j$.  This implies that the $d_i$ commute with each other.   A $q$-skew $\tau$-derivation $\delta$ on $R$ {\em extends to a h.$q$-s.$\tau$-d.} if there is a h.$q$-s.$\tau$-d $\lbrace d_i \rbrace$ on $R$ with $d_1 = \delta$. }
\end{definition}

For example, consider the $k$-algebra with two generators $x$ and $y$, and one relation $xy - qyx = 1$, where $q \in k^{\times}$.  We'll assume that $q \ne 1$ and recognize this algebra as a $q$-skew polynomial ring $k[y][x;\tau, \delta]$ with $\tau (y) = qy$ and $\delta (y) = 1$, commonly known as a {\em quantized Weyl algebra} and denoted $A_1^q(k)$.  If $q$ is not a root of unity, then the maps
\begin{equation}\label{qwasequence}
d_i = \frac{\delta^i}{(i)!_q}
\end{equation}
comprise an iterative higher $q$-skew $\tau$-derivation that extends $\delta$ on $k[y]$.  The properties of a higher $q$-skew $\tau$-derivation follow directly from the fact that $\delta$ is a $q$-skew $\tau$-derivation and the first $q$-Liebnitz rule. This particular h.$q$-s.$\tau$-d. is also locally nilpotent because
\begin{equation}\label{qwaderivation}
d_i (y^n) = \begin{cases}
\binom{n}{i}_q y^{n-i} \quad \text{when } i \le n,\\ 
0 \qquad \quad \quad \text{when } i > n.
\end{cases}
\end{equation}

\begin{proposition}\label{powseriesmap}
Let $\{ d_i \}$ be a sequence of $k$-linear maps on a $k$-algebra $R$ with \linebreak[4]$d_0 = {\rm id_R}$, and let $R[[x;\tau^{-1}]]$ be the skew power series ring where $\tau$ is a $k$-linear automorphism of $R$, the coefficients are written on the right of the variable $x$, and $rx=x\tau(r)$ for all $r \in R$.  

{\rm (a)} Then $\{ d_i \}$ is a higher $\tau$-derivation on $R$ if and only if the map $\Psi : R \rightarrow R[[x;\tau^{-1}]]$ given by $r \mapsto \sum_{i=0}^{\infty} x^i d_i(r)$ is a ring homomorphism.

{\rm (b)} Extend $\tau$ to an automorphism of $R[[x; \tau^{-1}]]$ such that $\tau(x) = xq$.  Assume that $\{ d_i \}$ is a higher $\tau$-derivation.  Then the sequence $\{ d_i \}$ is a h.$q$-s.$\tau$-d. if and only if this diagram is commutative:

$$\xymatrix{
R[[x; \tau^{-1} ]] \ar[r]^{\tau} & R[[x; \tau^{-1} ]] \\
R \ar[u]^{\Psi} \ar[r]^{\tau} & R \ar[u]_{\Psi}
}$$

\end {proposition}

\begin{proof}
(a) Suppose $\{d_i\}$ is a higher $\tau$-derivation on $R$.  Consider any $r, s \in R$.  It is clear that $\Psi$ is additive and $\Psi(1) = 1$.   Applying the definition \ref{highder} gives
\begin{equation*}
\Psi(rs) = \sum_{i=0}^{\infty} x^i d_i(rs) = \sum_{i=0}^{\infty} x^i \Big (\sum_{m=0}^i \tau^{i-m} d_m(r) d_{i-m}(s) \Big ).
\end{equation*}
Power series multiplication, with $rx = x\tau(r)$, gives
\begin{equation*}
\Psi(r)\Psi(s) = \Big ( \sum_{i=0}^{\infty} x^i d_i(r) \Big ) \Big ( \sum_{i=0}^{\infty} x^i d_i(s) \Big ) = \sum_{i=0}^{\infty} x^i \Big (\sum_{m=0}^i \tau^{i-m} d_m(r) d_{i-m}(s) \Big ).
\end{equation*}
So $\Psi$ preserves products.  Therefore, $\Psi$ is a ring homomorphism.

To demonstrate the other implication, suppose $\Psi$ is a ring homomorphism.  Then \linebreak[4] $\Psi(r) \Psi(s)= \Psi(rs)$ implies that $d_n(rs) = \sum_{i=0}^n \tau^{n-i} d_i(r) d_{n-i}(s)$ for all $r, s \in R$.  Therefore, $\{ d_i \}$ is a higher $\tau$-derivation.

(b) Suppose that $\{ d_i \}$ is a h.$q$-s.$\tau$-d.  Then the relations $d_i\tau = q^i \tau d_i$ imply that \linebreak[4]$\tau \Psi(r) = \sum_{i=0}^{\infty} x^i q^i \tau d_i(r) = \sum_{i=0}^{\infty} x^i  d_i (\tau(r)) = \Psi \tau (r)$, for all $r \in R$. 

Now if the diagram is commutative, then comparing the coefficients of\linebreak[4] $\tau \Psi(r) =   \sum_{i=0}^{\infty} x^i q^i \tau d_i(r)$ and  $\Psi \tau(r) = \sum_{i=0}^{\infty} x^i  d_i (\tau(r))$ for all $r \in R$ yields that \linebreak[4]$d_i\tau = q^i \tau d_i$.  \end{proof}

\begin{remark}
If $\{ d_i \}$ is locally nilpotent on $R$, we observe that claims analogous to the proposition can be made for the map $\Psi : R \rightarrow R[x;\tau^{-1}]$.
\end{remark}

\begin{proposition}\label{extendtoloc}
Let $\{ d_i \}$ be a h.$q$-s.$\tau$-d. on a $k$-algebra $R$, where $\tau$ is an automorphism, and let $S$ be a right denominator set in $R$ with $\tau(S) = S$.  Then $\{ d_i \}$ can be uniquely extended to a h.$q$-s.$\tau$-d. on $RS^{-1}$.
\end{proposition}

\begin{proof}
It has been established that $\tau$ and $d_1$ extend uniquely to $RS^{-1}$ by \linebreak[4]$\tau(rs^{-1}) = \tau(r) \tau(s)^{-1}$ and $d_1(rs^{-1}) = d_1(r)s^{-1} - \tau(rs^{-1})d_1(s)s^{-1}$ in \cite[Lemma 1.3]{primesprqwa}.  Suppose that $\{ d_i \}$ extends to a h.$q$-s.$\tau$-d. on $RS^{-1}$.  For $r \in R$ and $s \in S$, we apply $d_n$ to the equation $r1^{-1} = (rs^{-1})(s1^{-1})$ to get
\begin{align*}
d_n(r)1^{-1} &= d_n \big ( (rs^{-1})(s1^{-1}) \big ) = \sum_{j=0}^n \tau^{n-j} d_j (rs^{-1}) d_{n-j} (s1^{-1})\\
&=\tau^n(rs^{-1})d_n(s)1^{-1} + \dotsb + d_n(rs^{-1})s1^{-1}.
\end{align*}
This implies that
\begin{equation*}
d_n(rs^{-1}) = \Big [ d_n(r) - \sum_{j=0}^{n-1} \tau^{n-j} d_j(rs^{-1}) d_{n-j}(s) \Big ] s^{-1}.
\end{equation*}
So we have uniqueness in case of existence.

To show existence, let $\Psi : R \rightarrow R[[x; \tau^{-1} ]]$ be the map defined in Proposition \ref{powseriesmap}, and let $\phi : R[[x; \tau^{-1} ]] \rightarrow RS^{-1}[[x; \tau^{-1} ]]$ be the natural map.  Consider the composite map $\Phi = \phi \Psi: R \rightarrow RS^{-1}[[x; \tau^{-1} ]]$.  For any $s \in S$, the constant term of $\Phi(s)$ is a unit.  So we may inductively solve for the coefficients of an inverse for $\Phi(s)$ in $RS^{-1}[[x; \tau^{-1} ]]$.  Details, as in \cite[1.2]{rowen}, are left to the reader.  Hence, $\Phi$ extends to a ring homomorphism $\Phi' : RS^{-1} \rightarrow RS^{-1}[[x; \tau^{-1} ]]$ such that $\Phi' (rs^{-1}) = \Phi(r) \Phi(s)^{-1}$, and we consider the diagram:

$$\xymatrix{
RS^{-1}[[x; \tau^{-1} ]] \ar[r]^{\tau} & RS^{-1}[[x; \tau^{-1} ]] \\
RS^{-1} \ar[u]^{\Phi'} \ar[r]^{\tau} & RS^{-1} \ar[u]_{\Phi'}
}$$

where $\tau$ has been extended to an automorphism of $RS^{-1}[[x; \tau^{-1} ]]$ as in  Proposition \ref{powseriesmap}.

Since $\Phi(r) = \sum_{i=0}^{\infty} x^i d_i (r) 1^{-1}$, and $\{ d_i \}$ is a h.$q$-s.$\tau$-d. on $R$, we have
\begin{equation*}
\tau \Phi(r) = \sum_{i=0}^{\infty} x^i q^i \tau d_i (r) 1^{-1} = \sum_{i=0}^{\infty} x^i d_i \big ( \tau (r) \big ) 1^{-1} = \Phi \tau(r)
\end{equation*}
for all $r \in R$.  It follows directly that $\tau \Phi'(rs^{-1}) = \Phi' \tau(rs^{-1})$.  So, indeed, the diagram is commutative. 

Define a sequence $\{ d_i \}$ on $RS^{-1}$ such that $d_i(t)$ equals the coefficient of $x^i$ in $\Phi'(t)$ for all $t \in RS^{-1}$.  Then by Proposition \ref{powseriesmap} we conclude that this sequence  is a h.$q$-s.$\tau$-d. on $RS^{-1}$ extending $\{ d_i \}$ on $R$.
\end{proof}

\begin{lemma}\label{d-stablesubalg}
Let $A$ be a $k$-algebra, $B \subseteq A$ a $k$-subalgebra generated by $\lbrace b_1, b_2, \dotsc \rbrace$,  $\tau$ a $k$-linear automorphism of $A$,  and $\lbrace d_i \rbrace$ a higher $\tau$-derivation on $A$.  If $d_i(b_j) \in B$ and $\tau(b_j) \in B$, for all $i, j \in \mathbb N$, then $d_i(B) \subseteq B$ for all $i$.
\end{lemma}

\begin{proof}First, observe that $\tau(b_j) \in B$ for all $j$ implies that $\tau(B) \subseteq B$. Since the $d_i$ are $k$-linear maps, it suffices to check monomials in the $b_j$, using induction on their length.  
Suppose, inductively, that for integers $m \ge 1$ and $1 \le \ell \le m-1$, we have $d_i(b_{j_1} \dotsb b_{j_\ell}) \in B$ for all $i$ and all $j_1, \dotsc, j_{\ell}$.  Then using the product rule for h.$q$-s.$\tau$-d. gives

\begin{equation*}
d_n(b_{j_1} \dotsb b_{j_m}) = \sum_{i=0}^n \tau^{n-1}d_i(b_{j_1} \dotsb b_{j_{m-1} })d_{n-i}(b_{j_m}) \in B
\end{equation*}
for all $n$ and all $j_1, \dotsc, j_m$, by the induction hypothesis.
\end{proof}

\begin{lemma}\label{extendlocnilp}
Let $A$ be a $k$-algebra with a set $\{x_j \}$ of generators,  $\tau$ an automorphism of $A$, and $\{ d_i \}$ a h.$q$-s.$\tau$-d. on $A$.  If $\{d_i \}$ is locally nilpotent for all $x_j$, then $\{d_i \}$ is locally nilpotent on $A$.
\end{lemma}

\begin{proof}It suffices to check monomials in the $x_j$ because the $d_i$ are $k$-linear maps.  We proceed by using induction on the length of such monomials.  For a given $x_n$, let $i(n)$ be its nilpotence index, so $d_i (x_n) = 0$ for all $i \ge i(n)$.  

Suppose inductively that for $n \ge 2$, all integers $\ell$ with $1 \le \ell \le {n-1}$, and all choices of $j_1, \dotsc , j_{\ell}$, there exists an integer $m$ such that $d_i(x_{j_1} \dotsb x_{j_l}) = 0$ for all $i \ge m$. For instance, $m = i(j_1) + \dotsb + i(j_{\ell})$ will suffice, although the $d$-nilpotence index of $x_{j_1} \dotsb x_{j_{\ell}}$ may be less than this sum.   Then, for $p \ge m + i(j_n)$, we have 
\begin{equation*}
d_p(x_{j_1} \dotsb x_{j_n}) = \sum_{i=0}^p \tau^{p-i} d_i (x_{j_1} \dotsb x_{j_{n-1}}) d_{p-i}(x_{j_n}) = 0,
\end{equation*}
completing the induction.
\end{proof}

Consider again the quantized Weyl algebra $A_1^q(k)$.  In case $q$ is an $\ell$-th root of unity, the $d_{\ell}$ given in (\ref{qwaderivation}) would be undefined due to the occurrence of a zero denominator. However, realizing $A_1^q(k)$ as a factor of a quantized Weyl algebra over $k[t^{\pm 1}]$ allows us to define a h.$q$-s.$\tau$-d. on $A_1^q(k)$ nonetheless.  The $k[t^{\pm 1}]$-algebra $A_1^t(k[t^{\pm 1}])$ has generators $x$ and $y$ and one relation $xy - tyx = 1$.  This is a $t$-skew polynomial ring $k[t^{\pm 1}][y][x; \bar \tau, \bar \delta ]$ where $\bar \tau (y) = ty$, $\bar \tau (t) = t$, $\bar \delta (y) = 1$, and $\bar \delta (t) = 0$.  Note that 
\begin{equation*}
\bar \delta^i (y^n) = \begin{cases}
\frac{(n)!_t}{(n-i)!_t} y^{n-i} \quad \text{when } i \le n\\
0 \qquad \qquad \quad \text{when } i >n
\end{cases} 
\end{equation*}
implying that $\bar \delta^i \big ( k[t^{\pm 1}][y] \big ) \subseteq (i)!_t k[t^{\pm 1}][y]$.  So the assignment 
\begin{equation*} 
\bar d_i = \frac{\bar \delta^i}{(i)!_t}
\end{equation*}
defines an iterative, locally nilpotent h.$t$-s.$\bar \tau$-d. $\{ \bar d_i \}$ on $k[t^{\pm 1}][y]$.   Now, the relation $xy - tyx = 1$ is equivalent to the relation $xy - qyx = 1$ modulo $\langle t - q \rangle$.  Hence we have 
\begin{equation*}
A_1^t \big (k[t^{\pm 1}] \big ) / \langle t - q \rangle \cong A_1^q(k).
\end{equation*}
When $q$ is an $\ell$th root of unity, we have $\bar \delta^{\ell} \big (k[t^{\pm 1}][y] \big ) \subseteq \langle t - q \rangle k[t^{\pm 1}][y]$.  Nonetheless, the h.$t$-s.$\bar \tau$-d. $\{ \bar d_i \}$ on $k[t^{\pm 1}][y]$ induces a h.$q$-s.$\tau$-d. $\{ d_i \}$ on $k[y]$, also iterative and locally nilpotent, with $d_1 = \delta$.  Note that even though $\delta^{\ell}=0$ in this algebra, we have $d_i (y^i) = 1$ for all $i$.

This phenomenon is not unique to the quantized Weyl algebras.  The conditions that drive it are codified in the following theorem.

\begin{theorem}\label{specialize}
Let R be a $k$-algebra and $R[x;\tau, \delta]$ a $q$-skew polynomial ring where $q \in k$, $q \ne 1$.   Suppose there exists a torsion-free $k[t^{\pm 1}]$-algebra $\overline R$ and $\overline R[x;\bar \tau, \bar \delta]$ a $t$-skew polynomial ring such that $\overline R / \langle t-q \rangle \overline R \cong R$, with $\bar \tau$ and $\bar \delta$ reducing to $\tau$ and $\delta$.  Suppose further that $\bar \delta^i (\overline R) \subseteq (i)!_t \overline R$ for all $i$.  Then $\delta$ extends to an iterative h.$q$-s.$\tau$-d. $\{ d_i \}$ on $R$.  If $\bar \delta$ is locally nilpotent, then so is $\{ d_i \}$.  If $q$ is not a root of unity, then $d_i = \frac{\delta^i}{(i)!_q}$ for all $i$.  If $q$ is a primitive $\ell^{th}$ root of unity, then $d_i = \frac{\delta^i}{(i)!_q}$ for $i < \ell$.
\end{theorem}

\begin{proof}
The assumption $\bar \delta^i (\overline R) \subseteq (i)!_t \overline R$ for all $i$ implies that the sequence of maps \linebreak[4]$\bar d_i = \frac{\bar \delta^i}{(i)!_t}$ make up a well-defined iterative h.$t$-s.$\bar \tau$-d. on $\overline R$, and also implies that $\bar \delta^{\ell} (\overline R) \subseteq \langle t-q \rangle \overline R$ because 
$(\ell)_t \equiv (\ell)_q = 0$ modulo $\langle t-q \rangle$. Since $\bar \tau$ and $\bar \delta$ reduce to $\tau$ and $\delta$ modulo $\langle t-q \rangle$, we have an isomorphism $\overline R/ \langle t-q \rangle [x; \bar \tau, \bar \delta]  \cong R[x;\tau, \delta]$ whereby $\{ \bar d_i \}$ induces an iterative h.$q$-s.$\tau$-d. $\{ d_i \}$ on $R$.  The reduction of the maps from $\overline R$ to $R$ also implies the remaining results. 
\end{proof}

We will find that all of the conditions assumed above are satisfied by the common quantized coordinate rings and related examples, which will be discussed in a subsequent section.

\section{The $\tau$-Derivation Removing Homomorphism}
Following the pattern in \cite{cau}, let $A=R[x;\tau,\delta]$, and suppose that $\delta$ is locally nilpotent.    Set $S=\lbrace x^n \mid n\in \mathbb N \cup \{0\} \rbrace \subset A$.

\begin{lemma}\label{denomset} 
The set $S$ is a denominator set in $A$.
\end{lemma}

\begin{proof}
Clearly, $S$ is a multiplicative set in $A$.  And, since  $S$ contains only regular elements of $A$, it is left and right reversible.  It remains to show that $S$ is an Ore set.

Let $a=\sum_{i=0}^n r_i x^i$ be an element of $A$ with $r_n \ne 0$.
For each $r_i$ in the expression of $a$, and each $m_i \ge 0$, we have  
\begin{align*}
x^{m_i} r_i& = \sum_{j=0}^{m_i} \binom{m_i}{j}_q \tau^{m_i -j} \delta^j(r_i)x^{m_i -j} \\
  &= a'_ix+ \delta^{m_i}(r_i) \quad \text{for some  } a'_i\in A.
\end{align*}  

Since $\delta$ is locally nilpotent, we may choose $m_i$ to be the $\delta$-nilpotence index of $r_i$ to conclude that $x^{m_i}r_i= a'_ix$ for some $a'_i\in A$.  Set $m_a =\max\lbrace m_i \, \mid 0\le i\le n \rbrace$.  Then for each $r_i$, we have $x^{m_a}r_i=\tilde{a}_ix$, and hence $x^{m_a}a=\tilde{a}x$ for some $\tilde{a}\in A$.  

Now suppose, inductively, that for a given $a\in A$ and $x^p\in S$ we can find elements $x^{m_a}\in S$ and $\bar{a}\in A$ such that $x^{m_a}a=\bar{a}x^p$, say $\bar{a}=\sum_{i=0}^{n} \bar{r}_ix^i$. We know that there exists an element $x^{m_{\bar{a}} }$ such that $x^{m_{\bar{a}} }\bar{a}=a' x$ for some $a'\in A$.  So, $x^{m_a}a = \bar{a}x^p$ implies $x^{m_{\bar{a}} + m_a}a = a' x^{p+1}$, completing the induction. 

Hence, for any $a\in A$ and $s\in S$, we have $Sa\cap As \ne \emptyset$. So $S$ is a left Ore set in $A$.  We see that $S$ is a right Ore set by applying the same argument to $A^{\rm op} = R^{\rm op}[x;\tau^{-1},-\delta\tau^{-1}]. $
\end{proof}

Suppose also that the derivation $\delta$ extends to an iterative, locally nilpotent higher $q$-skew $\tau$-derivation $\lbrace d_i \rbrace$ on $R$ and that $q \ne 1$. Denote $\widehat A = AS^{-1} = S^{-1}A$, the localization of $A$ with respect to $S$, and define a map $f:R\xrightarrow{\quad}  \widehat{A} $ by 
\begin{equation*}
f(r)=\sum_{n=0}^\infty q^{\frac{n(n+1)}{2} }(q-1)^{-n} d_n \tau^{-n}(r)x^{-n},
\end{equation*}
noting that $\lbrace d_i \rbrace$ is locally nilpotent and that $q-1$ is invertible.  If $q$ is not a root of unity and $\{ d_i \}$ is obtained from a $q$-skew $\tau$-derivation $\delta$ as in (\ref{qwaderivation}), the formula for $f$ can be rewritten as
\begin{equation*}
f(r)=\sum_{n=0}^\infty q^{\frac{n(n+1)}{2} } \frac{(q-1)^{-n}}{(n)!_q}  \delta^n \tau^{-n}(r)x^{-n}.
\end{equation*}
The rewritten formula matches the one presented in \cite[Section 2]{cau} when $q$ is replaced by $q^{-1}$ to account for the difference between $\delta \tau = q \tau \delta$ (used here) and $\tau \delta = q \delta \tau$ (used in \cite{cau}). We will show that $f$ is a homomorphism and that the the multiplication in ${\rm im} f$ is made simpler than that in $A$ by removing the derivation, as seen in the following.

\begin{proposition}\label{frels}
If $r\in R$, then $x f(r)=f\bigl ( \tau(r) \bigr )x$ in $\widehat A$.
\end{proposition}

 \begin{proof} Using the hypothesis that $\{ d_i \}$ is iterative, we compute that
{\allowdisplaybreaks \begin{align*}
xf(r)&  = \sum_{n=0}^\infty q^{\frac{n(n+1)}{2} }(q-1)^{-n} x d_n \tau^{-n}(r)x^{-n} \\
& =\sum_{n=0}^\infty  q^{\frac{n(n+1)}{2} }(q-1)^{-n} \biggl[ \tau d_n \tau^{-n}(r) x+ d_1 d_n\tau^{-n}(r) \biggr] x^{-n} \\
& =\sum_{n=0}^\infty  q^{\frac{n(n+1)}{2} }(q-1)^{-n}q^{-n} d_n \tau^{-n+1}(r)x^{-n+1}\\
& \qquad \quad +\sum_{n=0}^\infty  q^{\frac{n(n+1)}{2} }(q-1)^{-n}(n+1)_q d_{n+1} \tau^{-n}(r)x^{-n} \\
& =\sum_{n=0}^\infty q^{\frac{n(n+1)}{2} }(q-1)^{-n} q^{-n}d_n \tau^{-n}( \tau(r) )x^{-n+1}\\
& \qquad \quad + \sum_{n=1}^\infty q^{\frac{n(n-1)}{2} }(q-1)^{-n+1}(n)_q d_n \tau^{-n}( \tau(r )) x^{-n+1} \\
& =\tau(r)x\\
& \qquad +\sum_{n=1}^\infty \biggl [  q^{\frac{n(n+1)}{2}} (q-1)^{-n} q^{-n}  +   q^{\frac{n(n-1)}{2}}  (q-1)^{-n+1} (n)_q\biggr ] d_n \tau^{-n}(\tau(r))x^{-n+1} \\
& =\tau(r)x\\
& \qquad + \sum_{n=1}^\infty (q-1)^{-n} \biggl [ q^{\frac{n^2-n}{2}}+q^{\frac{n^2-n}{2}}(q^n-1) \biggr ] d_n\tau^{-n}(\tau(r))x^{-n+1} \\
& =\tau(r)x + \sum_{n=1}^\infty (q-1)^{-n} q^{\frac{n(n+1)}{2}} d_n \tau^{-n}( \tau(r) )x^{-n+1} \\
& =\Biggl ( \sum_{n=0}^\infty  q^{\frac{n(n+1)}{2} }(q-1)^{-n} d_n \tau^{-n}( \tau(r) )x^{-n} \Biggr )x \, = \, f\bigl (\tau(r) \bigr )x, 
\end{align*}}
which gives the result.
\end{proof}

From Proposition \ref{frels}, it follows by routine induction that 
\begin{equation}\label{imagemult}
x^mf(r)=f\bigl (\tau^m(r) \bigr )x^m \quad \forall m \in \mathbb Z.
\end{equation}

This is what we need in order to show that our map is indeed a $k$-algebra homomorphism.

\begin{proposition} 
The map $f:R \xrightarrow{\quad} \widehat A$ is a $k$-algebra homomorphism.
\end{proposition}

\begin{proof} It is immediate that $f$ is $k$-linear ($\tau$ and $\lbrace d_i \rbrace$ are $k$-linear), and that $f(1) = 1$.  We'll show that $f$ is multiplicative. If $r,s\in R$, then using Prop. \ref{frels}, 
\begin{align*}
f(r)f(s)& =\sum_{i=0}^\infty q^{\frac{i(i+1)}{2} }(q-1)^{-i} d_i \tau^{-i}(r)x^{-i}f(s)\\
& =\sum_{i=0}^\infty q^{\frac{i(i+1)}{2} }(q-1)^{-i} d_i \tau^{-i}(r)f( \tau^{-i}(s) )x^{-i} \\\
& =\sum_{i \ge 0,\, j \ge 0}^{}  q^{\frac{i(i+1)+j(j+1)}{2} }(q-1)^{-(i+j)} d_i \tau^{-i}(r) d_j \tau^{-(i+j)}(s) x^{-(i+j)}.
\end{align*}
 
For $n\in \mathbb N$, the coefficient of $x^{-n}$ in the sum above is
{\allowdisplaybreaks \begin{align*}
c_n& =\sum_{\substack{i \ge 0,\, j \ge 0,\\ i+j =n}}^{} q^{\frac{i(i+1)+j(j+1)}{2} }(q-1)^{-n} d_i \tau^{-i}(r) d_j \tau^{-n}(s)\\
& =\sum_{p=0}^n q^{\frac{(n-p)^2+p^2+n}{2} } (q-1)^{-n} d_{n-p}\tau^{p-n}(r) d_{p}\tau^{-n}(s) \\
& = q^{\frac{n(n+1)}{2}}(q-1)^{-n} \sum_{p=0}^n q^{p(p-n)} d_{n-p} \tau^p \tau^{-n}(r) d_p \tau^{-n}(s) \\
& = q^{\frac{n(n+1)}{2}}(q-1)^{-n} \sum_{p=0}^n \tau^p d_{n-p}( \tau^{-n}(r) )d_p( \tau^{-n}(s)) \\
& = q^{\frac{n(n+1)}{2}}(q-1)^{-n} d_n( \tau^{-n}(r) \tau^{-n}(s) ) \\ 
& =  q^{\frac{n(n+1)}{2}}(q-1)^{-n} d_n\tau^{-n}(rs),
\end{align*}}
computed by putting $p = j$ and using the second condition in the Definition \ref{highder}.
In summary, $f(r)f(s)=\sum_{n=0}^\infty q^{\frac{n(n+1)}{2} }(q-1)^{-n} d_n \tau^{-n}(rs)x^{-n}=f(rs)$.
\end{proof} 

\begin{proposition}
 {\rm (1)} The map $f$ extends uniquely to an algebra homomorphism, also denoted $f$, of $R[y;\tau]$ to $\widehat A$ satisfying f(y)=x.

{\rm (2)} The extended homomorphism is injective.
\end{proposition}

\begin{proof} (1) This result follows from Proposition \ref{frels} and the universal property of Ore extensions.

(2) Let $P=p_m y^m+\dotsb +p_1y+p_0$ be a nonzero element of $R[y;\tau]$, where each $p_i\in R,\linebreak[4] \, m\ge 0, \, p_m\neq 0$.  Then $f(P)=f(p_m)x^m+\dotsb +f(p_1)x+f(p_0)$. Since 
\begin{equation*}
f(p_i)=\sum_{n=0}^\infty q^{\frac{n(n+1)}{2} }(q-1)^{-n} d_n \tau^{-n}(p_i)x^{-n}\, \in AS^{-1}, 
\end{equation*}
we know that there exists an integer $l\ge 0$ such that each $f(p_i)x^l$ is a nonzero element of A of positive degree $l$ (in $x$) whenever $p_i\neq 0$. (Because $\lbrace d_i \rbrace$ is locally nilpotent, we may choose an $l$ large enough.)  It follows that $f(P)x^l$ is a nonzero element of $\widehat A$ of degree $m+l$, hence $f(P)\neq 0$. 
\end{proof}

\begin{definition} 
{\rm The algebra homomorphism $f:R[y;\tau]\xrightarrow{\quad} \widehat A = AS^{-1}$ is called the {\em derivation removing homomorphism}. The image of $f$, call it $A^\prime$,  is the subalgebra of $\widehat A = AS^{-1}$ generated by $x$ and $f(R)$, and is isomorphic (as an algebra) to $R[y;\tau]$ by the derivation removing homomorphism $f$. }
\end{definition}

Observe that $A^\prime$ contains the multiplicative system $S=\lbrace x^n \mid n\in \mathbb N \cup \{0\} \rbrace$.  Since equation (\ref{imagemult}) holds and $f(y) = x$, the elements of this set are normal in $A^\prime$.  Hence, $S$   satisfies the (two-sided) Ore condition in $A^\prime$.  The elements of $S$ are regular in $A^\prime$ because they are regular in $\widehat A$, and thus:

\begin{proposition}\label{eqquotrings}
\quad $A^\prime S^{-1}=AS^{-1}$
\end{proposition}

\begin{proof} We have $A^\prime S^{-1}\subseteq AS^{-1}$ because $A^\prime = {\rm im}(f) \subseteq AS^{-1}$.  To show the other inclusion, it suffices to show that $R\subseteq A^\prime S^{-1}$. (This suffices because $A$ is built up from $R$ by $x,\, x^2, \dotsc$.  So if $R\subseteq A^\prime S^{-1}$, then $AS^{-1}\subseteq A^\prime S^{-1}$.)  Consider any $r\in R$ and let $\ell$ be the $d$-nilpotence index of $r$.
We show that $r\in A^\prime S^{-1}$ with an induction argument on $\ell$.

If $\ell\le 1$, then $d_1(r)=0$, whence $f(r)=r\in A^\prime \subseteq AS^{-1}$.

If $\ell\ge 2$, we write 
\begin{equation*}
f(r)=r+\sum_{n=1}^{\ell-1} r_n x^{-n}, \qquad \text{with } r_n= q^{\frac{n(n+1)}{2} }(q-1)^{-n} d_n\tau^{-n}(r)\in R.
\end{equation*}

We'll show that $\sum_{n=1}^{\ell-1} r_n x^{-n}\in A^\prime S^{-1}$ in order to conclude that $r \in A^\prime S^{-1}$, because $f(r)-\sum_{n=1}^{\ell-1} r_n x^{-n}=r$. That is, we need to show that each $r_n\in A^\prime S^{-1}$. 
Suppose, inductively, that for any element $\tilde{r}\in R$ with $d$-nilpotence index $m$ such that $m<\ell$, we have $\tilde{r}\in A^\prime S^{-1}$.

Note that for $n\in \lbrace 1,\dotsc ,\ell \rbrace$, we have 
\begin{equation*}
d_{\ell-n}(r_n)=q^{\frac{n(n+1)}{2} } (q-1)^{-n} \binom{\ell}{n}_q d_{\ell}\tau^{-n}(r) = q^{\frac{n(n+1)}{2} } (q-1)^{-n} \binom{\ell}{n}_q q^{-n\ell}\tau^{-n}d_\ell (r)=0
\end{equation*}
because $d_\ell (r)=0$ by hypothesis.

Hence, by the induction hypothesis, each $r_n \in A^\prime S^{-1}$ for $1\le n\le \ell-1$.  It follows that $r=f(r)-\sum_{n=1}^{\ell-1} r_nx^{-n}$ also belongs to $A^\prime S^{-1}$.
\end{proof}

This equality of quotient rings reveals that if $A$ is a PI ring, then
\begin{equation*}
{\rm PIdeg}\,A = {\rm PIdeg}\, A' = {\rm PIdeg}\, R[y;\tau],
\end{equation*}
with the second equality arising from the derivation removing homomorphism $f$.  This recovers the result of J\o ndrup \cite{jondrup} without the assumption that $k$ has characteristic zero. We summarize the results of this section in the following theorem.

\begin{theorem}\label{summary}
Let $k$ be a field, $R$ a $k$-algebra and $A= R[x; \tau, \delta]$ a $q$-skew polynomial ring in which $\delta$ extends to a locally nilpotent, iterative h.$q$-s.$\tau$-d. $\{d_i\}$ on $R$ for some \linebreak[4]$q \in k^{\times}, \, q \ne 1$.  Let $S$ be the Ore set in $A$ generated by $x$, and define a map \linebreak[4]$f: R \longrightarrow AS^{-1}$ by $f(r)= \sum_{n=0}^{\infty} q^{\frac{n(n+1)}{2}} (q-1)^{-n} d_n\tau^{-n}(r) x^{-n}$. Then $f$ is a $k$-algebra homomorphism, and it extends to an injective homomorphism $f: R[y; \tau] \longrightarrow AS^{-1}$ sending $y$ to $x$.  Furthermore, the extension $f: R[y^{\pm 1}; \tau] \longrightarrow AS^{-1}$ is an isomorphism.  So there is PI degree parity between $A$ and $R[y; \tau]$.  Moreover, if $R$ is a noetherian domain, then ${\rm Fract}\, A \cong {\rm Fract}\, R[y; \tau]$.
\end{theorem}

\section{Main Theorem}

In the case where $A$ is an iterated skew polynomial ring, we would like to apply repeatedly the method presented above to remove all of the derivations and compare the resulting Ore localizations.  We must first establish some facts about the behavior of h.$q$-s.$\tau$-d. when the variables adjoined to the coefficient ring are rearranged, and about iterated localization.  The results of these lemmas will ensure that after the induction step in the proof of the main theorem we are left with a ring to which the method of the preceding section applies.

The first parts of the following lemmas hold in a broader class of skew polynomial rings and also when the $q$-skew condition is imposed.  The final parts assert that h.$q$-s.$\tau$-d. are preserved when rearranging of the variables is permissible.

\begin{lemma}\label{switchingA}
Let $S=R[x;\tau, \delta]$, $A=R[x;\tau, \delta][y;\sigma]$, and $\widehat A = R[x;\tau, \delta][y^{\pm 1};\sigma]$, where $\sigma(R)=R$ and $\sigma(x)=\lambda x$ for some $\lambda \in k^{\times}$.

{\rm (1)} Then $A=R[y;\sigma'][x;\tau'; \delta']$, and $\widehat A = R[y^{\pm 1};\sigma'][x;\tau'; \delta']$, where $\sigma'=\sigma \bigl \vert_R$, $\tau' \bigl \vert_R =\tau$, $\delta' \bigl \vert_R =\delta$, $\tau'(y)=\lambda^{-1}y$, and $\delta'(y)=0$.

{\rm (2)} If $(\tau,\delta)$ is $q$-skew, then so is $(\tau', \delta')$. 

{\rm (3)} Suppose further that $\delta$ extends to a h.$q$-s.$\tau$-d. $\lbrace d_i \rbrace$ on $R$, and that  $\sigma d_i = \lambda^i d_i \sigma$ for all $i$. Then the $\tau'$-derivation $\delta'$ extends to a h.$q$-s.$\tau'$-d. $\lbrace d'_i \rbrace$ on $R[y^{\pm 1};\sigma']$ such that the restrictions of the $d'_i$ to $R$ coincide with $d_i$, and $d'_i (y)=0$ for all $i \ge 1$. Moreover, $\{ d'_i \}$ restricts to a h.$q$-s.$\tau'$-d. on $R[y; \sigma']$.

\quad {\rm (a)} If $\{d_i\}$ is iterative, then $\{d'_i\}$ is iterative.
     
\quad {\rm (b)} If $\{d_i\}$ is locally nilpotent, then $\{d'_i\}$ is locally nilpotent.

\end{lemma}

\begin{proof}
(1) Routine details omitted so as not to try the patience of the reader. 

(2)  Suppose that $(\tau,\delta)$ is $q$-skew on $R$.  We'll check that the two $\tau'$-derivations $\tau'^{-1} \delta' \tau'$ and $q \delta'$ agree on $R[y^{\pm 1}; \sigma']$.  It suffices to check their agreement on a set of generators, $R \cup \{y, y^{-1} \}$.  It is clear that $\tau'^{-1} \delta' \tau' (r) = q \delta' (r)$ for all $r \in R$.  Since $\delta' (y) = 0$, they agree on $\{y, y^{-1} \}$ as well.  So $(\tau', \delta')$ is $q$-skew.

(3) Define a sequence of maps $d'_i :R[y^{\pm 1}; \sigma'] \rightarrow R[y^{\pm 1}; \sigma']$ by
\begin{equation*}
d'_i \bigl ( \sum_{j=-m}^m r_j y^j \bigr ) = \sum_{j=-m}^m d_i(r_j) y^j.
\end{equation*}

Clearly these are $k$-linear maps, $d'_i(r)= d_i(r)$ for all $r\in R$;  also $d'_i(y) = d_i(1)y = 0$ for $i \ge 1$, and $d'_0$ is the identity on $R[y^{\pm 1}; \sigma']$.

Because $\delta$ extends to $\lbrace d_i \rbrace$ on $R$, we get
\begin{equation*}
d'_1(\sum_{j=-m}^m r_j y^j \bigr ) = \sum_{j=-m}^m d_1(r_j) y^j = \sum_{j=-m}^m \delta(r_j) y^j = \delta' \bigl (\sum_{j=-m}^m r_j y^j \bigr )
\end{equation*}
for all $r_j \in R$.   So $d'_1 = \delta'$ on $R[y^{\pm 1}; \sigma']$.

Now, for integers $j, m, n$, and elements $r, s \in R$,
\begin{align*}
d'_n\bigl ( (ry^j)(sy^m) \bigr ) &= d'_n \bigl ( r\sigma^j(s) y^{j+m} \bigr ) = d_n \bigl ( r\sigma^j(s) \bigr ) y^{j+m}\\
& = \sum_{i=0}^n \tau^{n-i}d_i(r)d_{n-i}(\sigma^j(s) )y^{j+m}\\
& = \sum_{i=0}^n \tau^{n-i}d_i(r)y^j  \sigma^{-j}d_{n-i}(\sigma^j(s) )y^m\\
& = \sum_{i=0}^n \tau^{n-i}d_i(r)y^j \lambda^{-j(n-i)}d_{n-i}(s) y^m\\
& = \sum_{i=0}^n (\tau')^{n-i}\bigl (d_i(r)y^j \bigr )d'_{n-i}(sy^m)\\
& = \sum_{i=0}^n (\tau')^{n-i}d'_i(ry^j)d'_{n-i}(sy^m).
\end{align*}

So  $\lbrace d'_i \rbrace$ satisfies the product rule for a higher $\tau$-derivation on $R[y^{\pm 1}; \sigma']$.

Furthermore, 
\begin{align*}
\tau' d'_i \bigl (\sum_{j=-m}^m r_j y^j \bigr ) &= \tau' \bigl (\sum_{j=-m}^m d_i(r_j) y^j \bigr) = \sum_{j=-m}^m \tau d_i(r_j)\lambda^{-j} y^j,\\
\text{and }\quad d'_i \tau' \bigl (\sum_{j=-m}^m r_j y^j \bigr ) &=  d'_i \bigl (\sum_{j=-m}^m \tau(r_j)\lambda^{-j} y^j \bigr ) = \sum_{j=-m}^m d_i \tau(r_j)\lambda^{-j}  y^j \\
& = q^i \sum_{j=-m}^m \tau d_i(r_j)\lambda^{-j} y^j,    
\end{align*}
giving the $q$-skew relation $d'_i \tau' = q^i \tau' d'_i$ on $R[y^{\pm 1}; \sigma']$.

It follows directly from the definition of the maps $\{ d_i \}$ that their restrictions to the $k$-subalgebra $R[y; \sigma']$ also exhibit the properties of definition \ref{highder}.

If $\lbrace d_i \rbrace$ is iterative on $R$, then $d'_{\ell} d'_i(ry^m) = d'_{\ell} \big ( d_i(r)y^m \big ) = d_{\ell} d_i(r)y^m = \binom{{\ell+i}}{i}_q d_{\ell+i} (r)y^m\\ =  \binom{{\ell+i}}{i}_q d'_{\ell+i} (ry^m)$ for all $r \in R$, $m \in \mathbb Z$, and non-negative integers $\ell, i$.  Hence, $\lbrace d'_i \rbrace$ is iterative on $R[y^{\pm 1}; \sigma']$.

Suppose that $\lbrace d_i \rbrace$ is locally nilpotent on $R$.  By Lemma \ref{extendlocnilp} we need only check that $\lbrace d'_i \rbrace$ is locally nilpotent on $R \cup \{y, y^{-1} \}$, a set of generators for $R[y^{\pm 1}; \sigma']$. This is clear because $d'_i(r) = d_i(r)$ for all $r \in R$, and $d'_i(y) = 0$ for all $i$ by construction.
\end{proof}

\begin{lemma}\label{switchingB}
Let 
\begin{align*}
A & =R[x_1;\tau_1, \delta_1][x_2;\tau_2, \delta_2] \dotsb [x_n;\tau_n, \delta_n][y;\sigma],\\
\widehat A & =R[x_1;\tau_1, \delta_1][x_2;\tau_2, \delta_2] \dotsb [x_n;\tau_n, \delta_n][y^{\pm 1};\sigma],
\end{align*}
where $\sigma(R)=R$, and for all $i \in \lbrace 1, \dotsc ,n \rbrace$, $\sigma(x_i)=\lambda_i x_i$ for some nonzero $\lambda_i \in k$.  Let
$A_j = R[x_1;\tau_1; \delta_1][x_2;\tau_2, \delta_2] \dotsb [x_j;\tau_j, \delta_j]$
for $j=1,2, \dotsc, n$, and $A_0=R$. 
 
{\rm (1)} Then 
\begin{align*}
A &= R[y;\sigma^*][x_1;\tau'_1, \delta'_1][x_2;\tau'_2, \delta'_2] \dotsb [x_n;\tau'_n, \delta'_n],\\
\widehat A &= R[y^{\pm 1};\sigma^*][x_1;\tau'_1, \delta'_1][x_2;\tau'_2, \delta'_2] \dotsb [x_n;\tau'_n, \delta'_n],
\end{align*}
where $\sigma^*=\sigma \big \vert_R$, $\tau'_i \big \vert_{A_j} =\tau_i$, $\delta'_i \big \vert_{A_j} =\delta_i$, $\tau'_i(y)=\lambda^{-1}_i y$, and $\delta'_i (y)=0$ for all $1 \le i \le n$ and $j \le i-1$.

{\rm (2)} If $(\tau_i,\delta_i)$ is $q_i$-skew for any $1 \le i \le n$, then $(\tau'_i, \delta'_i)$ is also $q_i$-skew.

{\rm (3)} Suppose that each $\delta_i$ extends to an h.$q_i$-s.$\tau_i$-d. $\lbrace d_{i,p} \rbrace_{p=0}^\infty$, and that $\sigma d_{i,p} = \lambda^p_i d_{i,p} \sigma$ on $A_{i-1}$ for all $i$ and $p$.  Then each $\delta'_i$ extends to a h.$q_i$-s.$\tau'_i$-d. $\lbrace d'_{i,p} \rbrace_{p=0}^\infty$ on the algebra $R\langle y, y^{-1}, x_1 , \dotsc, x_{i-1}    \rangle$, where $d'_{i,p}$ coincides with $d_{i,p}$ on $A_j$, for $j < i$, and $d'_{i,p}(y) = 0$ for $p \ge 1$. Moreover, $\{ d'_{i,p} \}$ restricts to a h.$q_i$-s.$\tau'_i$-d. on $R\langle y, x_1 , \dotsc, x_{i-1}    \rangle$.

\quad {\rm (a)} If $\{d_{i,p}\}$ is iterative for any $1 \le i \le n$, then $\{d'_{i,p}\}$ is iterative.
     
\quad {\rm (b)} If $\{d_{i,p}\}$ is locally nilpotent for any $1 \le i \le n$, then $\{d'_{i,p}\}$ is locally nilpotent.
\end{lemma}

\begin{proof}
(1) The condition $\sigma(x_i) = \lambda_i x_i$ for all $i$ implies that $\sigma(A_i) = A_i$.  We will use induction on $n$ to prove the result.

Lemma \ref{switchingA} proves the case $n=1$.  Suppose the result holds for all $m<n$, and consider $A = A_{n-1}[x_n; \tau_n, \delta_n][y;\sigma]$.  Application of Lemma \ref{switchingA}, and then the induction hypothesis, gives
\begin{align*}
A & = A_{n-1}[x_n; \tau_n, \delta_n][y;\sigma]\\
& = A_{n-1}[y; \sigma'][x_n; \tau'_n, \delta'_n]\\
& = R[x_1;\tau_1, \delta_1] \dotsb [x_{n-1}; \tau_{n-1}, \delta_{n-1}][y; \sigma'][x_n; \tau'_n, \delta'_n]\\
& = R[y; \sigma^*][x_1; \tau'_1, \delta'_1] \dotsb [x_n; \tau'_n, \delta'_n],
\end{align*}
with the desired conditions met by the automorphisms and derivations, completing the induction.  Similarly, $\widehat A = R[y^{\pm 1}; \sigma^*][x_1; \tau'_1, \delta'_1] \dotsb [x_n; \tau'_n, \delta'_n]$.

(2) Consider the two $\tau_i'$-derivations $\tau_i'^{-1} \delta_i' \tau_i'$ and $q_i \delta_i'$ on the ring 
\begin{equation*}
R[y^{\pm 1}; \sigma^*][x_1;\tau_1', \delta_1'] \dotsb [x_{i-1}; \tau_{i-1}', \delta_{i-1}']
\end{equation*}
for $1 \le i \le n$. Since $(\tau_i, \delta_i)$ is $q$-skew, it is clear that these two $\tau_i'$ derivations agree on $A_{i-1}$.  And since $\delta_i'(y) = 0$ for all $i = 1, \dotsc, n$, these two $\tau_i'$-derivations agree on a full set of generators of $R[y^{\pm 1}; \sigma^*][x_1;\tau_1', \delta_1'] \dotsb [x_{i-1}; \tau_{i-1}', \delta_{i-1}']$.  Hence, $\delta_i' \tau_i' = q_i \tau_i' \delta_i'$.

(3) Suppose the result holds for the algebra $R[x_1; \tau_1, \delta_1] \dotsb [x_{n-1}; \tau_{n-1}, \delta_{n-1}][y^{\pm 1}; \sigma]$.  Then Lemma \ref{switchingA} may be applied, with $A_{n-1}$ providing the coefficients, to get
\begin{equation*}
A_{n-1}[x_n; \tau_n, \delta_n][y^{\pm 1}; \sigma] = A_{n-1}[y^{\pm 1}; \sigma'][x_n; \tau'_n, \delta'_n],
\end{equation*}
where $\delta'_n$ extends to a h.$q_n$-s.$\tau'_n$-d. $\{d'_{n,p} \}$  on $A_{n-1}[y^{\pm 1}]$. The induction hypothesis gives the result.  
\end{proof}

\begin{definition}
{\rm For a $k$-algebra $A$ and $a,\, b \in A$, we say that $a$ and $b$} scalar commute {\rm if there is an element $\alpha \in k^{\times}$ such that $ab = \alpha ba$}.  {\rm We may also say that $a$ and $b$}\\ $\alpha$-commute.
\end{definition}

In the following two lemmas, we let $D$ denote the division ring of fractions for the noetherian domain $A$.  When comparing localizations of $A$, we identify them as subrings of $D$.

\begin{lemma}\label{localize}
Let $A$ be a noetherian domain, $S \subseteq A\setminus \{0\}$ an Ore set. Let $T$ be an Ore set in $AS^{-1} \setminus \{0\}$ with $S \subseteq T$. 

{\rm (1)} Then there exists an Ore set ${\widetilde T} \subseteq A \setminus \{0\}$ with $S \subseteq {\widetilde T}$ such that $A{\widetilde T}^{-1} = (AS^{-1})T^{-1}$.  

{\rm (2)}  Suppose A is a $k$-algebra and $S$ is generated by $s_1, \dotsc, s_n$ satisfying $s_i s_j = \gamma_{ij} s_j s_i$ for all $i, j$ and some $\gamma_{ij} \in k^{\times}$. Further suppose that $T$ is generated by $S \cup t$ for some $t \in AS^{-1}$ that satisfies $s_i t = \lambda_i ts_i$ for all $i$ and some $\lambda_i \in k^{\times}$.  Then there exist a cyclic Ore set $\widehat T \subseteq A\setminus \{0\}$ and an $(n+1)$-generator Ore set $\widehat S \subseteq A \setminus \{0\}$  such that $S \subseteq \widehat S$, and $(AS^{-1})T^{-1} = A {\widehat T}^{-1} = A {\widehat S}^{-1}$.

\end{lemma}

\begin{proof}(1)  Consider $T \cap A$, the subset in $T$ of elements with a denominator of 1.  Clearly, this is a multiplicative set in $A$ which contains $S$.  Set ${\widetilde T} = T \cap A$.     Let $a \in {\widetilde T}$ and $\alpha \in A$.    Then $a \in T$, and since $\alpha \in AS^{-1}$, there exist $b' \in T$ and $\beta' \in AS^{-1}$ such that $a \beta' = \alpha b'$.  By \cite[10.2]{bluebook}, there exist $y \in S$, and $b, \beta \in A$ such that $\beta' = \beta y^{-1}$ and $b' = by^{-1}$; hence, $a \beta y^{-1} = \alpha b y^{-1}$ in $AS^{-1}$.  It follows that $a \beta = \alpha b$ in $A$.  So ${\widetilde T}$ satisfies the right Ore condition in $A$, and the left Ore condition by symmetry.  By the universal property, $A{\widetilde T}^{-1} \cong (AS^{-1})T^{-1}$.  As subrings of $D$, we have $A{\widetilde T}^{-1} = (AS^{-1})T^{-1}$.

(2) The generating element $t$ has the form $t = \bar a (s_1^{m_1} s_2^{m_2} \dotsb s_n^{m_n})^{-1}$ for some $m_i \in \mathbb N$, and $\bar a \in A$.  For any $s_i \in S$, we have 
\begin{equation*}
s_i \bar a (s_1^{m_1} s_2^{m_2} \dotsb s_n^{m_n})^{-1}  = \lambda_i \bar a (s_1^{m_1} s_2^{m_2} \dotsb s_n^{m_n})^{-1} s_i = \mu \lambda_i \bar a s_i (s_1^{m_1} s_2^{m_2} \dotsb s_n^{m_n})^{-1},
\end{equation*}
where $\mu$ is a product of powers of the $\gamma_{ij}$.  So $\bar a$ scalar commutes with the generators of $S$ via the relations $s_i \bar a = \mu \lambda_i \bar a  s_i$.  Let $\widehat S$ be the multiplicative set generated by $\bar a, s_1, \dotsc, s_n$ in $A$, and $\widehat T$ the multiplicative set generated by $\bar a s_1 s_2 \dotsb s_n$ in $A$. 
Recall that $(AS^{-1})T^{-1} = A{\widetilde T}^{-1}$, where $\widetilde T = T \cap A$ from part (1).  From the scalar commuting relations it follows that any element $a {\tilde t}^{-1} \in A {\widetilde T}^{-1}$ may be written in the form $b(\bar a s_1, \dotsb s_n)^{-m}$ for some $m \in \mathbb N \cup \{0\}, \; b \in A$, or the form $c{\bar a}^{-\ell_{n+1}} s_1^{-\ell_1} \dotsb s_n^{-\ell_n}$, for $\ell_j \in \mathbb N \cup \{0\}, \; c \in A$.  So we conclude that $\widehat S$ and $\widehat T$ are Ore sets in $A$ and that  $(AS^{-1})T^{-1} = A {\widehat T}^{-1} = A {\widehat S}^{-1}$.
\end{proof}

\begin{lemma}\label{localizeB}
Let $A$ be a noetherian domain, $S_1 \subseteq A\setminus \{0\}$ an Ore set, and for  integers $j = 2, \dotsc, n$ let $S_j$ be an Ore set in $((AS_1^{-1}) \dotsb )S_{j-1}^{-1}\setminus \{0\}$ with $S_{j-1} \subseteq S_j$. 

 {\rm (1)} Then there exists an Ore set $T \subseteq A\setminus \{0\}$ such that $AT^{-1} = (((AS_1^{-1})S_2^{-1}) \dotsb )S_n^{-1}$. 

{\rm (2)} Suppose $A$ is a $k$-algebra, $S_1$ is generated by $s_1$, and for $j = 2, \dotsc, n$, $S_j$ is generated by $S_{j-1} \cup \{s_j\}$, where $s_i s_j = \gamma_{ij} s_j s_i$ for some multiplicatively antisymmetric matrix $(\gamma_{ij}) \in M_n(k^{\times})$.  Then there are a cyclic Ore set $\widehat T \subseteq A$ and an $n$-generator Ore set $\widehat S \subseteq A$ such that $S_1 \subseteq \widehat S$, and $((AS_1^{-1}) S_2^{-1} ) \dotsb S_n^{-1} = A{\widehat T}^{-1} = A{\widehat S}^{-1}$.
\end{lemma}

\begin{proof}(1) The proof proceeds by induction on $n$. The case $n=1$ is covered in the lemma above.  Suppose that for all $j \le n-1$ there exists an Ore set $T_j \subseteq A\setminus \{0\}$ such that $AT_j^{-1} = (((AS_1^{-1})S_2^{-1}) \dotsb )S_j^{-1}$. Then the equality 
\begin{equation*}
AT_{n-1}^{-1} = (((AS_1^{-1})S_2^{-1}) \dotsb )S_{n-1}^{-1}
\end{equation*}
identifies an Ore set $T_n \subseteq AT_{n-1}^{-1}\setminus \{0\}$ such that 
\begin{equation*}
(AT_{n-1}^{-1})T_n^{-1} = (((AS_1^{-1})S_2^{-1}) \dotsb S_{n-1}^{-1})S_n^{-1}.
\end{equation*}
Furthermore, Lemma \ref{localize} implies the existence of an Ore set $T \subseteq A\setminus \{0\}$ such that $AT^{-1} = (AT_{n-1}^{-1})T_n^{-1} = (((AS_1^{-1})S_2^{-1}) \dotsb S_{n-1}^{-1})S_n^{-1}$. 

(2) Suppose, inductively, that there exist
\begin{itemize}
\item[] (i) a cyclic Ore set ${\widehat T}_{n-1} \subseteq A \setminus \{0\}$ generated by $s_1 {\bar a}_2 \dotsb {\bar a}_{n-1}$  
\item[] (ii) an $(n-1)$-generator Ore set ${\widehat S}_{n-1} \subseteq A \setminus \{0\}$ with $S_1 \subseteq \widehat S_{n-1}$ and generators $s_1,\; {\bar a}_2,\; {\bar a}_3, \dotsc,\; {\bar a}_{n-1}$
\item[] (iii) the ${\bar a}_i$ scalar commute with $s_1$ and with each other  
\item[] (iv) $((AS_1^{-1}) S_2^{-1}) \dotsb S_{n-1}^{-1} = A{\widehat T}_{n-1}^{-1} = A{\widehat S}_{n-1}^{-1}$ as subrings of $D$. 
\end{itemize} 
Then $s_n = {\bar a}_n(s_1{\bar a}_2 \dotsb {\bar a}_{n-1})^{-r}$ for some ${\bar a}_n \in A$ and $r \in \mathbb N$.  Using the relations \linebreak[4]$s_i s_j = \gamma_{ij} s_j s_i$, routine calculations show that the ${\bar a}_i$ scalar commute with the $s_j$, and also with each other, for all $i, j$.  Let $\widehat T$ be the multiplicative set generated by $s_1{\bar a}_2 \dotsb {\bar a}_n$, and let $\widehat S$ be the multiplicative set generated by $s_1,\; {\bar a}_2,\; {\bar a}_3, \dotsc,\; {\bar a}_n$.  Then $((AS_1^{-1})S_2^{-1}) \dotsb S_n^{-1} = (A{\widehat T}_{n-1}^{-1} )S_n^{-1} = AT^{-1}$ from part (1).  Using Lemma \ref{localize}, we conclude that $\widehat T$ and $\widehat S$ are Ore sets in $A$ and that $AT^{-1} = A{\widehat T}^{-1} = A{\widehat S}^{-1}$.
\end{proof}

In the proof of the main theorem, we will use without mention the facts gathered here.  For greater details on these statements, see \cite[10X, 10Y]{bluebook} and \cite[1.4]{primesprqwa}.  
\begin{itemize}
\item[(1)] Given a noetherian ring $A$ and a normal element $x \in A$, the multiplicative set generated by $x$ is an Ore set.
\item[(2)] The multiplicative set generated by a nonempty family of right Ore sets is right Ore.
\item[(3)] Let $A = R[x;\tau, \delta]$, and $S$ a right denominator set in $R$ such that $\tau(S) = S$.  Then $S$ is a right denominator set in $A$ and the identity map on $AS^{-1}$ extends to an isomorphism of $AS^{-1}$ onto $(RS^{-1})[x; \tau, \delta]$ sending $x1^{-1}$ to $x$.  Note that if $A$ is a $k$-algebra, $\tau$, $\delta$ are $k$-linear, and $\tau(k^{\times}S) = k^{\times}S$, then the result holds because $S$ is a denominator set if and only if $k^{\times}S$ is a denominator set.
\end{itemize}

\begin{theorem}\label{mainthm}
Let $R$ be a $k$-algebra and noetherian domain, 
\begin{equation*}
A = R[x_1 ; \tau_1, \delta_1] \dotsb [x_n; \tau_n, \delta_n],\\
\end{equation*}
where each $\tau_i$ is a $k$-linear automorphism of $R \langle x_i, \dotsc, x_{i-1} \rangle$ such that $\tau_i(x_j) = \lambda_{ij} x_j$ for all $i, j$ with $1 \le j < i \le n$ and some $\lambda_{ij} \in k^{\times}$, and where each $\delta_i$ is a  $k$-linear $\tau_i$-derivation.  Assume that there exist elements $q_i \in k^{\times}$ with $q_i \ne 1$ such that $\delta_i \tau_i = q_i \tau_i \delta_i$, and that $\delta_i$ extends to a locally nilpotent, iterative h.$q_i$-s.$\tau_i$-d. on $R \langle x_i, \dotsc, x_{i-1} \rangle$ for $i=1, \dotsc, n$. 

{\rm (1)} Then there exists an Ore set $T \subseteq A$ generated by $n$ elements of $A$ such that 
\begin{equation*}
AT^{-1} \cong R[y_1^{\pm 1}; \tau_1][y_2^{\pm 1}; \tau_2'] \dotsb [y_n^{\pm 1}; \tau_n']
\end{equation*}
where $\tau_i' \vert_R = \tau_i$ and $\tau_i'(y_j) = \lambda_{ij}y_j$ for all $i, j$ with $1 \le j < i \le n$

{\rm (2)} There is PI degree parity between $A$ and $R[y_1; \tau_1][y_2; \tau_2'] \dotsb [y_n; \tau_n']$.  Moreover, these algebras have isomorphic division rings of fractions.
\end{theorem}

\begin{proof} (a) Suppose, inductively, that we have 
\begin{equation*}
R[x_1 ; \tau_1, \delta_1] [y_2^{\pm 1}; \tau_2] \dotsb [y_n^{\pm 1}; \tau_n'] \cong AS_2^{-1}
\end{equation*}
where the restriction of $\tau_i'$ to $R \langle x_1 \rangle $ coincides with $\tau_i$,  $\tau_i'(y_m) = \lambda_{im} y_m$ for $2 \le i \le n$ and $1 < m < i$, and $S_2$ is an Ore set in $A$ generated by $n-1$ elements from $A$.  Then by Lemma \ref{switchingB}
\begin{equation}\label{locisom}
AS_2^{-1} \cong R[y_2^{\pm 1};\tau_2''] \dotsb [y_n^{\pm 1}; \tau_n''][x_1; \tau_1', \delta_1'] 
\end{equation}
where the restrictions of $\tau_1'$ and $\delta_1'$ to $R$ coincide with $\tau_1$ and $\delta_1$, $\tau'_1(y_j) = \lambda_{j1}^{-1}y_j$, $\delta_1'(y_j) = 0$, and $\tau_i''$ coincides with the restriction of $\tau_i$ to $R \langle y_2, \dotsc, y_{i-1} \rangle$ for $2 \le i \le n$.  Observe that by Lemmas \ref{switchingB} and \ref{extendlocnilp} we also have $\delta_1' \tau_1' = q_1 \tau_1' \delta_1'$, and that $\delta_1'$ extends to a locally nilpotent iterative h.$q_1$-s.$\tau$-d. on $R \langle y_2^{\pm 1}, \dotsc, y_n^{\pm 1} \rangle$.  Then applying the derivation removing homomorphism to the right hand side of $\eqref{locisom}$ gives an isomorphism
\begin{equation*}
(AS_2^{-1})T_1^{-1} \cong R[y_2^{\pm 1}; \tau_2'] \dotsb [y_n^{\pm 1}; \tau_n'][y_1^{\pm 1}; \tau_1']
\end{equation*}
where $T_1 \subseteq AS_2^{-1}$ is an Ore set generated by one element of $AS_2^{-1}$.  Then Lemma \ref{localizeB} and a reordering of variables shows the existence of an Ore set $T \subseteq A$, generated by $n$ elements of $A$, such that $AT^{-1} \cong R[y_1^{\pm 1}; \tau_1][y_2^{\pm 1}; \tau_2'] \dotsb [y_n^{\pm 1}; \tau_n']$.

(2) This follows from part (1).
\end{proof}   

\begin{corollary}\label{qaffine}
Let $A = k[x_1 ; \tau_1, \delta_1] \dotsb [x_n; \tau_n, \delta_n]$ with the hypotheses as in Theorem  {\rm{\ref{mainthm}}}.  Set ${\boldsymbol \lambda} = (\lambda_{ij})$.  Then

{\rm (1)} $A$ and $\mathcal O_{\boldsymbol \lambda} (k^n)$ have isomorphic division rings of fractions.

{\rm (2)} $A$ is a PI-algebra if and only if all the $\lambda_{ij}$ are roots of unity, in which case $A$ and $\mathcal O_{\boldsymbol \lambda} (k^n)$ have the same PI degree.
\end{corollary}

In general, identification of the generators for the Ore set $T$ in Theorem \ref{mainthm} is very cumbersome.  To illustrate the computations on a fairly short iterated skew polynomial ring, we consider the multiparameter second quantized Weyl algebra $A_2^{Q,\Gamma}(k)$.  Here, $Q = (q_1, q_2) \in (k^{\times})^2$, $q_i \ne 1$ for all $i$, and $\Gamma = (\gamma_{ij}) \in M_2(k^{\times})$ with $\gamma_{ii} = 1$ and $\gamma_{21} = \gamma_{12}^{-1}$.   The algebra $A_2^{Q,\Gamma}(k)$ may be presented as an iterated skew polynomial ring of the form $k[y_1][x_1; \tau_2, \delta_2][y_2; \tau_3][x_2; \tau_4, \delta_4]$, where the $\tau_i$ are $k$-linear automorphisms and the $\delta_{2i}$ are $k$-linear $\tau_{2i}$-derivations such that
\begin{align*}
& \tau_2(y_1) = q_1 y_1, \qquad  \qquad  \delta_2(y_1) = 1\\
& \tau_3(y_1) = \gamma_{12}^{-1} y_1\\
& \tau_3(x_1) = \gamma_{12} x_1\\
& \tau_4(y_1) = q_1 \gamma_{12} y_1, \qquad  \quad \delta_4(y_1) = 0\\
& \tau_4(x_1) = q_1^{-1} \gamma_{21} x_1, \qquad \delta_4(x_1) = 0\\
& \tau_4(y_2) = q_2 y_2, \qquad \qquad \delta_4(y_2) = (q_1-1)y_1 x_1 + 1.
\end{align*}
For greater detail about this algebra, the reader is referred to \cite{Alev}, \cite{jordan}, \cite{catenarity}, and \cite{DMequiv}.  Routine computations show that the pair $(\tau_2, \delta_2)$ is a $q_1$-skew derivation and that $(\tau_4, \delta_4)$ is a $q_2$-skew derivation.  To show that $\delta_2$ and $\delta_4$ are locally nilpotent, it suffices to check for local nilpotence on a set of generators.  Given their definitions, this is accomplished by verifying their action on powers of $y_1$ and $y_2$:
\begin{equation*}
\delta_2^i (y_1^n) = \begin{cases}
\frac{(n)!_{q_1}}{(n-i)!_{q_1}} y_1^{n-i} \qquad i \le n\\
0 \qquad \qquad \qquad i>n
\end{cases}
\end{equation*}
\begin{equation*}
\delta_4^i(y_2^n) = \begin{cases}
\frac{(n)!_{q_2}}{(n-i)!_{q_2}}[\delta_4(y_2)]^i y_2^{n-i} \qquad i \le n\\
0 \qquad \qquad \qquad \qquad \quad i>n
\end{cases}
\end{equation*}
Using Theorem \ref{specialize} we have a h.$q_1$-s.$\tau_2$-d. $\{ d_{2,i} \}$ extending $\delta_2$, and a h.$q_2$-s.$\tau_4$-d. $\{ d_{4,i} \}$ extending $\delta_4$, both of which are iterative and locally nilpotent.  Let $S_2 \subseteq A_2^{Q,\Gamma}(k)$ be the multiplicative set generated by $x_2$. The derivation removing homomorphism induces an isomorphism 
\begin{equation*}
\Phi: k[y_1][x_1; \tau_2, \delta_2][y_2; \tau_3][z_2^{\pm 1}; \tau_4] \longrightarrow A_2^{Q,\Gamma}(k)S_2^{-1}
\end{equation*}
whose action on generators is given by
\begin{align*}
& y_1 \mapsto y_1\\
& x_1 \mapsto x_1\\
& z_2 \mapsto x_2\\
& y_2 \mapsto y_2 + (q_2 - 1)^{-1} \Bigl( (q_1-1)y_1 x_1 + 1 \Bigr) x_2^{-1}.
\end{align*}
For simplicity, label the domain of $\Phi$ as $BZ^{-1}$. Let $X_1 \subseteq BZ^{-1}$ be the Ore set generated by $z_2$ and $x_1$. Applying the derivation removing homomorphism to $BZ^{-1}$ induces an isomorphism
\begin{equation*}
\Psi: k[y_1][z_1^{\pm 1}; \tau_2][y_2; \tau_3][z_2^{\pm 1}; \tau_4'] \longrightarrow (BZ^{-1})X_1^{-1}
\end{equation*}
whose action on generators is given by
\begin{align*}
& z_1 \mapsto z_1\\
& z_2 \mapsto z_2\\
& y_2 \mapsto y_2\\
& y_1 \mapsto y_1 + (q_1-1)^{-1}x_1^{-1}.
\end{align*}
The derivation removing homomorphism need not be employed again to achieve the result.  Through iterated localization we find that there is an Ore set $T \subseteq A_2^{Q,\Gamma}(k)$ such that
\begin{equation*}
A_2^{Q,\Gamma}(k)T^{-1} \cong k[y_1^{\pm 1}][x_1^{\pm 1}; \tau_2][y_2^{\pm 1}; \tau_3][x_2^{\pm 1}; \tau_4]
\end{equation*}
and $T$ is generated by the four elements $x_2$, $x_1$, $y_2 x_2 (q_2 - 1) + y_1 x_1 (q_1 - 1) +1$,  and  $y_1 x_1(q_1 - 1) + 1$.  Note that we recover the result of \cite[Theorem 5]{jondrup2}.

\section{Examples}

We will demonstrate how each of the following $k$-algebras satisfies all the conditions of Theorem \ref{specialize}.  Then Corollary \ref{qaffine} is applied to obtain an isomorphism of quotient division rings (thereby confirming the quantum Gel'fand-Kirillov conjecture) and PI degree parity with a multiparameter quantum affine space.  When calculating the PI degree of a quantum affine space, we encounter an antisymmetric, or {\em skew-symmetric}, integral matrix. As proved in \cite[Theorem IV.1]{Newman}, such a matrix is congruent to a matrix in {\em skew normal form}.

\begin{theorem}\label{newman}
{\rm [Newman]} Let $A$ be a skew-symmetric matrix of rank $r$ which belongs to $M_n(R)$, where the commutative principal ideal domain $R$ is not of characteristic 2.  Then $r=2s$ and $A$ is congruent to a matrix in block diagonal form
\begin{equation*}S =
\begin{pmatrix}
 0     & h_1 &         &          &              &         & & & & \\
-h_1& 0     &          &          &              & &\boldsymbol 0 & & & \\
        &         & 0      & h_2  &              &         & & & &\\
        &         & -h_2 & 0      &              &         & & & &\\
        &         &          &          &  \ddots &         &         & & & &\\
        &         &          &          &              & 0      & h_s & & & &\\
        &\boldsymbol 0         &      & &  & -h_s & 0     & & & &\\
        &         &          &          &             &           &        & & & &\\        
        &         &          &          &              &          &        & & \boldsymbol 0& &\\
        &         &          &          &              &          &        &  & & &        
\end{pmatrix}
\end{equation*}
where $h_i \mid h_{i+1}, \, 1 \le i \le s-1$.
\end{theorem}

The same result, in the language of alternating bilinear forms, can be found in \cite [Section 5.1]{Bourbaki}.

The matrix $S$ in Theorem \ref{newman} is clearly equivalent to the more familiar Smith normal form, diag$(h_1,h_1,h_2,h_2,\dotsc,h_s,h_s,0,0,\dotsc,0)$, where the diagonal entries are the invariant factors of the matrix $A$.  In the examples that follow, we outline the operations necessary to obtain the Smith normal form.

\begin{definition}  {\rm Let $A = k[x_1;\tau_1, \delta_1] \dotsb [x_n;\tau_n, \delta_n]$ and $A' = k[x_1;\tau_1] \dotsb [x_n;\tau_n]$ be iterated skew polynomial rings. 
(1)  If there exists $Q = (q_1, \dotsc , q_n) \in (k^{\times})^n$ such that $\delta_i \tau_i = q_i \tau_i \delta_i$ for $i = 1, \dotsc, n$, then $A$ is called an} iterated $Q$-skew polynomial ring.  {\rm (2) If there exist $\lambda_{ji} \in k^{\times}$ such that $\tau_j(x_i) = \lambda_{ji} x_i$ for all $i<j$, then set $\lambda _{ij} = \lambda_{ji}^{-1}$ and $\lambda_{ii} = 1$ for all $i$.  We call $\Lambda = (\lambda_{ij}) \in M_n(k^{\times})$} the matrix of relations for $A'$.
\end{definition}

\begin{lemma}\label{containment}
Let $C$ be a commutative $k$-algebra, $A$ a $C$-algebra, $B \subseteq A$ a $C$-subalgebra generated by $\{ b_1, b_2, \dotsc \}$. Let $\tau$ be a $C$-algebra automorphism of $A$, and $\delta$ a $u$-skew $\tau$-derivation on $A$ for some unit $u \in C$.  If $\tau(b_j) \in B$ and $\delta^n (b_j) \in (n)!_u B$ for all $j,\,n$, then $\delta^n (B) \subseteq (n)!_u B$ for all $n$.
\end{lemma}

\begin{proof}
Note that $\tau (b_j) \in B$ for all $j$ implies that $\tau (B) \subseteq B$ and hence we have \linebreak[4]$\tau \big ( (j)!_u B \big ) \subseteq (j)!_u B$ for all $j$.  Suppose that for integers $m \ge 1$ and $1 \le \ell \le m-1$, we have $\delta^i (b_{j_1} \dotsb b_{j_{\ell}} ) \in (i)!_u B$ for all $i$, and all choices of $j_1, \dotsc, j_{\ell}$.  Then
\begin{align*}
\delta^n (b_{j_1} \dotsb b_{j_m}) &= \sum_{i=0}^n \binom{n}{i}_u \tau^{n-i} \delta^i (b_{j_1} \dotsb b_{j_{(m-1)}} ) \delta^{n-i} (b_{j_m})\\
& \in \sum_{i=0}^n \binom{n}{i}_u (i)!_u (n-i)!_u B  \subseteq (n)!_u B
\end{align*}
for all $n$ and all $j_1, \dotsc, j_m$ by induction.
\end{proof}

For a first family of examples, we take odd-dimensional quantum Euclidean spaces.  The even-dimensional ones will be covered in Example 5.4.

\subsection{The coordinate ring of odd-dimensional quantum Euclidean space; $\mathcal O_q (\mathfrak o k^{2n+1})$}

For $q \in k^{\times}$, assuming $q$ has a (fixed) square root $q^{1/2} \in k$, the $k$-algebra $\mathcal O_q (\mathfrak o k^{2n+1})$ may be presented as an iterated skew polynomial ring 
\begin{equation*} 
k[w][y_1; \sigma_1][x_1; \tau_1, \delta_1] \dotsb [y_n; \sigma_n][x_n; \tau_n, \delta_n]
\end{equation*}
with automorphisms $\sigma_i, \tau_i$ and derivations $\delta_i$ defined by
\begin{align*}
\sigma_i (w) & = q^{-1} w &\text{all } i\\
\tau_i (w) & = qw & \text{all } i\\
\sigma_i (y_j) & =q^{-1}y_j & j<i\\
\sigma_i (x_j) & = q^{-1}x_j  & j<i\\
\tau_i (y_j) & = qy_j  & i \ne j\\
\tau_i (x_j) & =qx_j  &j<i\\
\tau_i (y_i) & = y_i  &\text{all } i\\
\delta_i (w) & = \delta_i (x_j) = \delta_i (y_j) = 0  & j<i\\
\delta _i (y_i) & = (q^{1/2} - q^{3/2})w^2 + (1 - q^2) \sum_{\ell < i}y_{\ell} x_{\ell} & \text{all }i.
\end{align*}
Quantum Euclidean spaces have been studied since 1990 when they were introduced by Reshetikhin et al. in \cite{RTF}.  The three-dimensional case has applications to the structure of space-time at small distances.  Musson simplified the original set of relations in \cite{musson}, and Oh further simplified them, renaming the generators $\omega,\, x_i, \, y_i$ in \cite{oh}.  Here, we have made a change to Oh's variables, $y_i \mapsto q^i y_i$, to obtain the relations in our presentation of  $\mathcal O_q (\mathfrak o k^{2n+1})$.

Routine computations show that $\tau_i^{-1} \delta_i \tau_i (y_i) = q^{-2}\delta_i(y_i)$ for all $i$, and so we conclude that each $(\tau_i, \delta_i)$ is a $q^{-2}$-skew derivation. We may present the analogous $k[t^{\pm 1}]$-algebra $\mathcal O_t (\mathfrak o k[t^{\pm 1}]^{2n+1})$ as an iterated skew polynomial ring with coefficient ring $k[t^{\pm 1}]$ and generators $w$, $y_i$, $x_i$ for $i=1, \dotsc, n$, 
\begin{equation*}
k[t^{\pm 1}][w][y_1; \bar \sigma_1][x_1; \bar \tau_1, \bar \delta_1] \dotsb [y_n; \bar \sigma_n][x_n; \bar \tau_n, \bar \delta_n]
\end{equation*}
where the automorphisms and derivations are defined analogously to those of the algebra $\mathcal O_q (\mathfrak o k^{2n+1})$ with $t \in k[t^{\pm 1}]$ replacing $q \in k^{\times}$.  So each $(\bar \tau_i, \bar \delta_i)$ is a $t^{-2}$-skew derivation.  It is immediate that 
\begin{equation*} 
\mathcal O_t (\mathfrak o k[t^{\pm 1}]^{2n+1}) / \langle t-q \rangle \cong \mathcal O_q (\mathfrak o k^{2n+1})
\end{equation*}
with each $\bar \tau_i$ and $\bar \delta_i$ reducing to $\tau_i$ and $\delta_i$ respectively.

Let $A_j$ denote the $k[t^{\pm 1}]$-subalgebra generated by $w$, $y_m$, $x_m$ for $m<j$, and $y_j$.  To show that $\bar \delta_j^i (A_j) \subseteq (i)!_{t^{-2}} A_j$, we apply Lemma \ref{containment} noting that $\bar \delta_j^i (y_j)$ has been given for $ i=1$ and is zero for $i>1$.
So, by Theorem \ref{specialize}, each $\delta_i$ in our presentation of $\mathcal O_q (\mathfrak o k^{2n+1})$ extends to an iterative, locally nilpotent h.$q^{-2}$-s.$\tau_i$-d. on an appropriate subalgebra. Then Corollary \ref{qaffine} gives 
\begin{equation*}
{\rm Fract}\,\mathcal O_q (\mathfrak o k^{2n+1}) \cong {\rm Fract}\, \mathcal O_B (k^{2n+1}),
\end{equation*}
where the matrix of relations is 
\begin{equation*}
B = \begin{pmatrix} 1 & q & q^{-1} & q & q^{-1} & \cdots &q & q^{-1}\\ 
q^{-1} & 1 & 1 & q & q^{-1} &\cdots &q &q^{-1}\\ 
q & 1 & 1 & q & q^{-1} & \cdots & q & q^{-1}\\
q^{-1} & q^{-1} & q^{-1} & 1 & 1 & \cdots & q & q^{-1}\\
q & q & q & 1 & 1 & \cdots & q & q^{-1}\\ 
\vdots & \vdots & \vdots & \vdots & \vdots &\ddots  & \vdots & \vdots \\ 
 q^{-1} & q^{-1} & q^{-1} & q^{-1} & q^{-1} & \cdots& 1 & 1 \\ 
 q & q & q & q & q & \cdots & 1 & 1 \end{pmatrix}.
\end{equation*} 
If $q \in k^{\times}$ is a root of unity, we may assume without loss of generality that it is a primitive $r^{th}$ root of unity.  Then the powers of $q$ from the matrix $B$ become the entries of a  $(2n+1) \times (2n+1)$ integer matrix
\begin{equation*}
B' = \begin{pmatrix} 0 & 1 & -1 & 1 & {-1} & \cdots &1 & {-1}\\ 
{-1} & 0 & 0 & 1 & {-1} &\cdots &1 &{-1}\\ 
1 & 0 & 0 & 1 & {-1} & \cdots & 1 & {-1}\\
{-1} & {-1} & {-1} & 0 & 0 & \cdots & 1 & {-1}\\
1 & 1 & 1 & 0 & 0 & \cdots & 1 & {-1}\\ 
\vdots & \vdots & \vdots & \vdots & \vdots &\ddots  & \vdots & \vdots \\ 
 {-1} & {-1} & {-1} & {-1} & {-1} & \cdots& 0 & 0 \\ 
 1 & 1 & 1 & 1 & 1 & \cdots & 0 & 0 \end{pmatrix}.
\end{equation*} 
Now, ${\rm PI deg}\, \mathcal O_q (\mathfrak o k^{2n+1})$ can be computed from Theorem \ref{dp}(2) using the matrix $B'$.  The cardinality of the image will not be changed if we first perform some row reductions on $B'$.  Letting $N = 2n+1$, $n > 2$, we manipulate the rows as follows.
\begin{itemize}
\item For $i=2,4,6,\dotsc,N-1$, replace row $i$ with row $i$ + row $(i+1)$.
\item For $i=N,N-2,N-4,\dotsc,5$, replace row $i$ with row $i$ $-$ row $(i-2)$.
\item Replace row 5 with row 5 $-$ row 1.
\item For $i=2,4,6,\dotsc,N-5$, replace row $i$ with row $i$ $-$ 2row $(i+5)$.
\item Multiply the even numbered rows, except row $2n-2$, by $-1$.
\end{itemize}
The resulting matrix has $2n$ pivots and one zero row.  We put the rows in this order
\begin{equation*}
3,1,5,7,2,9,4,11,6,13,\dotsc,2i,2i+7,\dotsc,N,N-5,N-3,N-1
\end{equation*}
to place the pivots on the main diagonal and the zero row in the last position.  Then we have a matrix of this form
\begin{equation*}
\begin{pmatrix} 
1 & * & * & * & * & * & * & * & * & *  & * & *\\ 
0 & 1 & {-1} & * & * & * & * & * & * & * & * & *\\
0 & 0 & 2 & * & * & * & * & * & * & * & * & *\\ 
0 & 0 & 0 & 1 & 1 & * & * & * & * & * & * & *\\ 
0 & 0 & 0 & 0 & 4 & * & * & * & * & * & * & *\\
0 & 0 & 0 & 0 & 0 & 1 & 1 & * & * & * & * & *\\
0 & 0 & 0 & 0 & 0 & 0 & 4 & * & * & * & * & *\\
0 & 0 & 0 & 0 & 0 & 0 & 0 &\ddots & * & * & * & * \\
0 & 0 & 0 & 0 & 0 & 0 & 0 & 0 & 1 & 1 & * & *\\ 
0 & 0 & 0 & 0 & 0 & 0 & 0 & 0 & 0 & 4 & * & *\\ 
0 & 0 & 0 & 0 & 0 & 0 & 0 & 0 & 0 & 0 & 2 & -2\\ 
0 & 0 & 0 & 0 & 0 & 0 & 0 & 0 & 0 & 0 & 0 & 0
\end{pmatrix}.
\end{equation*}
The diagonal entries of this echelon matrix do not yet reveal the size of its image because the pivot in row three does not divide all of the (suppressed) entries in its row when $n \ge 3$.  So more row reduction is needed.
\begin{equation*}
\text{First replace row 3 with } \text{row } 3 + \sum_{i=1}^{\lfloor \frac{n-2}{2} \rfloor} \text{row}(4i+2).\\
\end{equation*}
For $n$ even and  $j = 5, 7, 9, \dotsc, 2n-3$,  replace row $j$ as follows:
\begin{align*}
&\text{for } j=4p+1, p \ge 1, \text{ use row } j + \sum_{i=p+1}^{\frac{n-2}{2}}2\cdot \text{row}(4i) + \text{row}(2n);\\
&\text{for } j=4p+3, p \ge 1, \text{ use row } j + \sum_{i=p+1}^{\frac{n-2}{2}}2\cdot \text{row}(4i+2).
\end{align*}
For $n$ odd and $j = 5, 7, 9, \dotsc, 2n-5$, replace row $j$ as follows:
\begin{align*}
&\text{for } j=4p+1, p \ge 1, \text{ use row } j + \sum_{i=p+1}^{\lfloor \frac{n}{2}\rfloor}2\cdot \text{row}(4i) + 2\cdot \text{row}(2n);\\
&\text{for } j=4p+3, p \ge 1, \text{ use row } j + \sum_{i=p+1}^{\lfloor \frac{n-2}{2}\rfloor}2\cdot \text{row}(4i+2) + \text{row}(2n).
\end{align*}
Then add row($2n$) to row($2n-3$), and add 2$\cdot$row($2n$) to row($2n-1$).  For integers \linebreak[4]$4 \le j \le 2n-1$, with $ j \not \equiv 2(\text{mod } 4)$, add $(-1)^j \text{col } 3$ to col $j$.  Subtract col($2n+1$) from col 3; add row 3 to row($2n-2$); and subtract 2$\cdot$row 3 from row($2n$).  The result is an upper echelon matrix in which each pivot divides all the nonzero entries in its row.  So it is trivial to diagonalize by column operations.  The Smith normal form for $n$ odd is diag$(1,1,\dotsc,1,4,4,\dotsc,4,0)$ with $n+1$ ones and $n-1$ fours.  The Smith normal form for $n$ even is diag$(1,1,\dotsc,1,2,2,4,4,\dotsc,4,0)$ with $n$ ones, two twos, and $n-2$ fours.

For the cases $ n = 1, \, 2$, the row-reduced matrices are, respectively,
\begin{equation*}
\begin{pmatrix}
1 & 0 & 0\\ 0 & 1 & -1 \\ 0 & 0 & 0
\end{pmatrix}, \qquad
\begin{pmatrix}
1 & 0 & 0 & 1 & -1 \\ 0 & 1 & -1 & 1 & -1\\ 0 & 0 & 2 & -2 & 2\\ 0 & 0 & 0 & 2 & -2\\ 0 & 0 & 0 & 0 & 0
\end{pmatrix}.
\end{equation*} 
Hence we have, for all $n>0$, 
\begin{equation*}  
{\rm PI deg}\,  \mathcal O_q (\mathfrak o k^{2n+1})= \begin{cases}
r^n, \qquad \qquad  &r \text{ odd}\\
r^n /  2^{\lfloor \frac{n}{2} \rfloor}, \qquad  &r \text{ even, } r \notin 4\mathbb Z\\
r^n / 2^{n-1}, \qquad   &r \in 4\mathbb Z
\end{cases}.
\end{equation*}

\subsection{The multiparameter quantized Weyl algebras; $A_n^{Q, \Gamma}(k)$}

For a fixed $n$-tuple $Q = (q_1, \dotsc, q_n) \in (k^{\times})^n$ and $\Gamma = (\gamma_{ij})$ a multiplicatively antisymmetric $n \times n$ matrix over $k$, the algebra $A_n^{Q, \Gamma}(k)$, studied in \cite{jordan} and \cite{malt}, may be presented as an iterated skew polynomial ring
\begin{equation*} 
k[y_1][x_1; \tau_1, \delta_1][y_2; \sigma_2][x_2; \tau_2, \delta_2] \dotsb [y_n; \sigma_n][x_n; \tau_n, \delta_n]
\end{equation*}
where the automorphisms and derivations are defined by
\begin{align*}
\sigma_i (y_j) & = \gamma_{ji} y_j  &j<i\\
\sigma_i (x_j) & = \gamma_{ij} x_j & j<i\\
\tau_i (y_j) & = q_j \gamma_{ji} y_j & j<i\\
\tau_i (x_j) & = q_j^{-1} \gamma_{ij} x_j &j<i\\
\tau_i (y_i) & = q_i y_i & \text{all } i\\
\delta_i (x_j) & = \delta_i (y_j) = 0 & j<i\\
\delta_i (y_i) & = 1 + \sum_{\ell <i} (q_{\ell} - 1) y_{\ell} x_{\ell} & \text{all } i. 
\end{align*}
Routine computations show that $\tau_i^{-1} \delta_i \tau_i (y_i) = q_i \delta_i (y_i)$ for all $i$, and so we conclude that each $(\tau_i, \delta_i)$ is a $q_i$-skew derivation. We may present the $k[t_1^{\pm 1}, \dotsc, t_n^{\pm 1}]$-algebra $A_n^{T, \Gamma}(k[t_1^{\pm 1}, \dotsc, t_n^{\pm 1}])$ as an iterated skew polynomial ring 
\begin{equation*}
k[t_1^{\pm 1}, \dotsc, t_n^{\pm 1}][y_1][x_1;\bar \tau_1, \bar \delta_1][y_2;\bar \sigma_2][x_2;\bar \tau_2, \bar \delta_2] \dotsb [y_n;\bar \sigma_n][x_n;\bar \tau_n, \bar \delta_n]
\end{equation*}
where the automorphisms and derivations are defined analogously to those of $A_n^{Q, \Gamma}(k)$ with $t_i \in k[t_1^{\pm 1}, \dotsc, t_n^{\pm 1}]$ replacing $q_i \in k$.  So each $(\bar \tau_i, \bar \delta_i)$ is a $t_i$-skew derivation.  It is immediate that
\begin{equation*}  
A_n^{T, \Gamma}(k[t_1^{\pm 1}, \dotsc, t_n^{\pm 1}]) /  \langle t_1-q_1, \dotsc, t_n-q_n \rangle  \cong A_n^{Q, \Gamma}(k)
\end{equation*}
with each $\bar \tau_i$ and $\bar \delta_i$ reducing to $\tau_i$ and $\delta_i$ respectively.

Let $A_j$ denote the $k[t_1^{\pm 1}, \dotsc, t_n^{\pm 1}]$-subalgebra generated by $y_m$, $x_m $ for $m<j$, and $y_j$.  To show that $\bar \delta_j^i(A_j) \subseteq (i)!_{t_j} A_j$, it suffices to check $\bar \delta_j^i(y_j)$ by Lemma \ref{containment}.  But this is given by definition for $i=1$ and is zero for $i>1$. 
So, by Theorem \ref{specialize}, each $\delta_i$ in our presentation of $A_n^{Q, \Gamma}(k)$ extends to an iterative, locally nilpotent h,$q_i$-s.$\tau_i$-d. on the appropriate subalgebra.  Then Corollary \ref{qaffine} gives ${\rm Fract}\, A_n^{Q, \Gamma}(k) \cong {\rm Fract}\, \mathcal O_{\Lambda}(k^{2n})$, where the $2n \times 2n$ matrix of relations $\Lambda$ is comprised of $2 \times 2$ blocks

\begin{align*}
B_{ii} & = \begin{pmatrix} 1 &q_i^{-1}\\ q_i &1\end{pmatrix}, \text{ for all } i; \\B_{ij} &= \begin{pmatrix} \gamma_{ji} &q_i^{-1} \gamma_{ji}\\ \gamma_{ij} &q_i \gamma_{ij} \end{pmatrix}, \text{ for } i<j ;\\B_{ij} &= \begin{pmatrix} \gamma_{ji} & \gamma_{ij}\\ q_j \gamma_{ji} &q_j^{-1}\gamma_{ij} \end{pmatrix}, \text{ for } i>j.
\end{align*}
If $\gamma_{ij}$ and $q_i$ are roots of unity for all $i, j$, then $\mathcal O_{\Lambda}(k^{2n})$ is a PI algebra.  Assuming that $\gamma_{ij}$ is an $r_{ij}^{th}$ root of unity and that $q_i$ is an $r_i^{th}$ root of unity, we let 
\begin{equation*}
r = {\rm lcm}\{ r_{ij}, r_i \mid i, j = 1, \dotsc, n\}.
\end{equation*}
Then there exists a primitive $r^{th}$ root of unity $q \in k$ and integers $b_i$, $b_{ij}$ such that $q_i = q^{b_i}$ and $\gamma_{ij} = q^{b_{ij}}$ for $i, j = 1, \dotsc, n$.  The powers of this $q$ from the matrix $\Lambda$ give a $2n \times 2n$ integer matrix $\Lambda'$ comprised of $2 \times 2$ blocks

\begin{align*}
B'_{ii} & = \begin{pmatrix} 0 &-b_i\\ b_i &0\end{pmatrix}, \text{ for all } i; \\B'_{ij} &= \begin{pmatrix} b_{ji} &b_{ji}-b_i\\ b_{ij} &b_{ij}+b_i \end{pmatrix}, \text{ for } i<j ;\\B'_{ij} &= \begin{pmatrix} b_{ji} & b_{ij}\\ b_j+ b_{ji} &b_{ij}-b_j \end{pmatrix}, \text{ for } i>j.
\end{align*}
Then ${\rm PI deg}\, A_n^{Q. \Gamma} (k)$ can be computed using the matrix $\Lambda'$ in Theorem \ref{dp} (2).

Consider the single parameter case, denoted $A_n^q(k)$, where $q_i = q$ for all $i$, and $\gamma_{ij} = 1$ for $i<j$, relegating the $\sigma_i$ to identity maps.  Assuming that $q$ is a primitive $r^{th}$ root of unity, then $\delta_i(y_i^r) = 0$ and $\tau_i(y_i^r) = y_i^r$ for all $i$, implying that $y_i^r$ is central.  The definition of the $\tau_i$, along with the $q$-Liebnitz rule, implies that $x_i^r$ is central for all $i$.  So the algebra $A_n^q(k)$ is a finitely generated module over the central subring $k[y_i^r,\,x_1^r, \dotsc, y_n^r,\, x_n^r]$.   To find the PI degree in this case, the integer matrix becomes
\begin{equation*}
\Lambda' = \begin{pmatrix}
0 &-1  &0 & -1 & \hdots & 0 & -1\\
1 & 0 & 0 & 1 &  & 0 & 1\\
0 & 0 & 0 & -1 & & 0 & -1\\
1 & -1 & 1 & 0 & & 0 & 1\\
\vdots & & & & \ddots & & \\
0 & 0 & 0  &0 & \hdots & 0 & -1\\
1 & -1 & 1 & -1 & \hdots & 1 & 0
\end{pmatrix},
\end{equation*}
which is seen to have a trivial kernel after these row reductions:
\begin{itemize}
\item Replace row $2n$ with $\text{row } 2n - \text {row } (2n-2) - \text{row } (2n -3)$
\item For $j = n-1, n-2, \dotsc, 2$, replace row $2j$ with $\text{row }2j - \text{row }(2j-2) - \text{row }(2j-3)$
\item Rearrange the rows to order $2, 1, 4, 3, 6, 5 \dotsc, 2n, 2n-1$.
\end{itemize}
The resulting matrix has the form
\begin{equation*}
 \begin{pmatrix}
1 & 0  & &  &  &  & \\
 0 & -1  &  &  &  & \boldsymbol{*}  & \\
  &   & 1 & 0 & & & \\
  &   & 0 & -1 & & & \\
 &  &  &  & \ddots & & \\
 &  \bf 0 &   &  &  & 1 & 0\\
 &  &  &  & & 0 & -1
\end{pmatrix},
\end{equation*}
thus verifying that ${\rm PI deg}\, A_n^q(k) = r^n$.

\subsection{The multiparameter coordinate ring of quantum $n \times n$ matrices; $\mathcal O_{\lambda, {\boldsymbol p}}\big (M_n(k)\big )$}

The multiparameter coordinate ring of quantum $n \times n$ matrices was introduced by Artin, Schelter, and Tate in \cite{ast}.
The $k$-algebra $\mathcal O_{\lambda, {\boldsymbol p}}\big (M_n(k)\big )$ is defined by generators $x_{ij}$ for $i,j = 1, \dotsc, n$ and relations 

\begin{equation*}
x_{\ell m} x_{ij}  = \begin{cases}
p_{\ell i} p_{jm} x_{ij} x_{\ell m} + (\lambda - 1)p_{\ell i} x_{im} x_{lj} \phantom{XX} (\ell> i, \, m > j)\\
\lambda p_{\ell i} p_{jm} x_{ij} x_{\ell m} \phantom{XXXXXXXXXxx} (\ell > i, \, m \le j)\\
p_{jm}x_{ij}x_{\ell m} \phantom{XXXXXXXXXXXxx} (\ell = i, \, m > j),
\end{cases}
\end{equation*}
where $\lambda \in k^{\times}$ and ${\boldsymbol p} = (p_{ij}) \in M_{n^2}(k^{\times})$ is multiplicatively antisymmetric.  It can also be presented as an iterated skew polynomial ring
\begin{equation*}
k[x_{11}][x_{12};\tau_{12}] \dotsb [x_{ij};\tau_{ij}, \delta_{ij}] \dotsb [x_{nn};\tau_{nn}, \delta_{nn}]
\end{equation*}
where each $\tau_{\ell m}$ and $\delta_{\ell m}$ is $k$-linear and satisfies

\begin{equation*}
\tau_{\ell m}(x_{ij}) = \begin{cases}
p_{\ell i} p_{jm}x_{ij} \phantom{XXXXXx} \text{when } \ell > i \text{ and } m \ne j\\
\lambda p_{\ell i} p_{jm} x_{ij} \phantom{XXXXX} \text{when } \ell > i \text{ and } m=j\\
p_{jm} x_{ij} \phantom{XXXXXXX} \text{when } \ell = i \text{ and } m>j,
\end{cases}
\end{equation*}

\begin{equation*}
\delta_{\ell m}(x_{ij}) = \begin{cases}
(\lambda - 1) p_{\ell i} x_{im} x_{\ell j}  \phantom{XXX} \text{when } \ell > i \text{ and } m>j\\
0 \phantom{XXXXXXXXXX} \text{otherwise.}
\end{cases}
\end{equation*}
Routine computations show $\tau_{\ell m}^{-1} \delta_{\ell m} \tau_{\ell m} (x_{ij}) = \lambda^{-1} \delta_{\ell m}(x_{ij})$ as in \cite[Section 5]{uniranks}, and so we conclude that each $(\tau_{\ell m}, \delta_{\ell m})$ is a $\lambda^{-1}$-skew derivation.  We may present the $k[t^{\pm 1}]$-algebra $\mathcal O_{t, {\boldsymbol p}}\big (M_n(k[t^{\pm 1}])\big )$ as an iterated skew polynomial ring with generators $x_{ij}$ for $i, j = 1, \dotsc, n$
\begin{equation*}
k[t^{\pm 1}][x_{11}][x_{12}, \bar\tau_{12}] \dotsb [x_{ij}; \bar \tau_{ij}, \bar \delta_{ij}] \dotsb [x_{nn}; \bar \tau_{nn}, \bar \delta_{nn}]
\end{equation*}
where the automorphisms and derivations are defined analogously to those of the algebra $\mathcal O_{\lambda, {\boldsymbol p}}\big (M_n(k)\big )$ with $t \in k[t^{\pm 1}]$ replacing $\lambda \in k$.  So each $(\bar \tau_{\ell m}, \bar \delta_{\ell m})$ is a $t^{-1}$-skew derivation.  It is immediate that 
\begin{equation*}
\mathcal O_{t, {\boldsymbol p}}\big (M_n(k[t^{\pm 1}])\big ) / \langle t -\lambda \rangle \cong \mathcal O_{\lambda, {\boldsymbol p}}\big (M_n(k)\big ) 
\end{equation*}
with each $\bar \tau_{\ell m}$ and $\bar \delta_{\ell m}$ reducing to   $\tau_{\ell m}$ and $\delta_{\ell m}$ respectively.  

Let $A_{\ell m}^-$ denote the $k[t^{\pm 1}]$-subalgebra generated by the $x_{ij}$ with $(i, j) < (\ell, m)$ in the lexicographic order.  Lemma \ref{containment} allows us to to verify that $\bar \delta_{\ell m}^s (A_{\ell m}^- ) \subseteq (s)!_{t^{-1}} (A_{\ell m}^-)$ by checking only that $\bar \delta_{\ell m}^s (x_{ij})$ is contained in $A_{\ell m}^-$.  This is immediate from the formula for $\bar \delta_{\ell m}$ given above. 
Thus, by Theorem \ref{specialize}, each $\delta_{\ell m}$ in our presentation of $\mathcal O_{\lambda, {\boldsymbol p}}\big (M_n(k)\big )$ extends to an iterative, locally nilpotent h.$\lambda^{-1}$-s.$\tau_{\ell m}$-d. on the appropriate $k$-subalgebra.  Then Corollary \ref{qaffine} gives 
\begin{equation*}
{\rm Fract}\, \mathcal O_{\lambda, {\boldsymbol p}} \big (M_n(k) \big ) \cong {\rm Fract}\, \mathcal O_{\Lambda} (k^{n^2})
\end{equation*}
where the matrix of relations $\Lambda = (b_{ij}) \in M_{n^2} (k)$ is comprised of $n \times n$ blocks
\begin{equation*}
B_{ii}= \begin{pmatrix} 1 &p_{21} &p_{31} &\cdots &p_{n1}\\ p_{12} &1 &p_{32} &\cdots &p_{n2}\\ p_{13} &p_{23} &1 &\cdots &p_{n3}\\ \vdots &\vdots &&\ddots &\vdots\\ p_{1n} &p_{2n} &p_{3n} &\cdots &1 \end{pmatrix} \text{  for all } i,
\end{equation*}
\begin{equation*}
B_{ij}= \begin{pmatrix} \lambda^{-1}p_{ij} &p_{ij}p_{21} &p_{ij}p_{31} &\cdots &p_{ij}p_{n1}\\ \lambda^{-1}p_{ij}p_{12} &\lambda^{-1}p_{ij} &p_{ij}p_{32} &\cdots &p_{ij}p_{n2}\\ \lambda^{-1}p_{ij}p_{13} &\lambda^{-1}p_{ij}p_{23} &\lambda^{-1}p_{ij} &\cdots &p_{ij}p_{n3}\\ \vdots &\vdots &&\ddots &\vdots\\ \lambda^{-1}p_{ij}p_{1n} &\lambda^{-1}p_{ij}p_{2n} &\lambda^{-1} p_{ij} p_{3n} &\cdots &\lambda^{-1}p_{ij} \end{pmatrix}, \text{  for } i<j,\\
\end{equation*} 
\begin{equation*}
B_{ij}= \begin{pmatrix} \lambda p_{ij} &\lambda p_{ij}p_{21} &\lambda p_{ij}p_{31} &\cdots &\lambda p_{ij}p_{n1}\\ p_{ij}p_{12} &\lambda p_{ij} &\lambda p_{ij}p_{32} &\cdots &\lambda p_{ij}p_{n2}\\ p_{ij}p_{13} &p_{ij}p_{23} & \lambda p_{ij} &\cdots &\lambda p_{ij}p_{n3}\\ \vdots &\vdots &&\ddots &\vdots\\  p_{ij}p_{1n} & p_{ij}p_{2n} &p_{ij} p_{3n} &\cdots &\lambda p_{ij} \end{pmatrix}, \text{  for } i>j.
\end{equation*}
If $\lambda$ and $p_{ij}$ are roots of unity for all $i, j$, then $\mathcal O_{\Lambda}(k^{n^2})$ is a PI algebra.  In this case we may assume that $\lambda$ is an $s^{th}$ root of unity and that $p_{ij}$ is an $r_{ij}^{th}$ root of unity, and let \linebreak[4]r = ${\rm lcm}\{s, r_{ij} \mid i, j = 1, \dotsc, n \}$.  Then there exists a primitive $r^{th}$ root of unity $q \in k$ and integers $b, b_{ij}$ such that $\lambda = q^b$ and $p_{ij} = q^{b_{ij}}$.  The powers of this $q$ from the matrix $\Lambda$ provide entries for an $n^2 \times n^2$ integer matrix $\Lambda'$ made up of $n \times n$ blocks 
\begin{equation*}
B'_{ii}= \begin{pmatrix} 0 &b_{21} &b_{31} &\cdots &b_{n1}\\ b_{12} &0 &b_{32} &\cdots &b_{n2}\\ b_{13} &b_{23} &0 &\cdots &b_{n3}\\ \vdots &\vdots &&\ddots &\vdots\\ b_{1n} &b_{2n} &b_{3n} &\cdots &0 \end{pmatrix} \text{  for all } i,
\end{equation*}
\begin{equation*}
B'_{ij}= \begin{pmatrix} b_{ij}-b &b_{ij}+b_{21} &b_{ij}+b_{31} &\cdots &b_{ij}+b_{n1}\\ b_{ij}+b_{12}-b &b_{ij}-b &b_{ij}+b_{32} &\cdots &b_{ij}+b_{n2}\\ b_{ij}+b_{13}-b &b_{ij}+b_{23}-b &b_{ij}-b &\cdots &b_{ij}+b_{n3}\\ \vdots &\vdots &&\ddots &\vdots\\ b_{ij}+b_{1n}-b &b_{ij}+b_{2n}-b &b_{ij}+b_{3n}-b &\cdots &b_{ij}-b \end{pmatrix}, \text{  for } i<j,\\
\end{equation*} 
\begin{equation*}
B'_{ij}= \begin{pmatrix} b_{ij}+b &b_{ij}+b_{21}+b &b_{ij}+b_{31}+b &\cdots &b_{ij}+b_{n1}+b\\ b_{ij}+b_{12} &b_{ij}+b &b_{ij}+b_{32}+b &\cdots &b_{ij}+b_{n2}+b\\ b_{ij}+b_{13} &b_{ij}+b_{23} & b_{ij}+b &\cdots &b_{ij}+b_{n3}+b\\ \vdots &\vdots &&\ddots &\vdots\\ b_{ij}+b_{1n} & b_{ij}+b_{2n} &b_{ij}+b_{3n} &\cdots &b_{ij}+b \end{pmatrix}, \text{  for } i>j.
\end{equation*}
Then ${\rm PIdeg}\, \mathcal O_{\lambda, {\boldsymbol p}} \big ( M_n (k) \big )$ can be calculated using $\Lambda'$ in Theorem \ref{dp} (2).

The single parameter quantized coordinate ring of $n \times n$ matrices, $\mathcal O_q(M_n(k))$, is defined over $k$ analogously to $\mathcal O_{\lambda, \boldsymbol p}(M_n(k))$, but with relations that are recovered by setting \linebreak[4]$\lambda = q^{-2}$ and $p_{ij} = q$ for all $i>j$.  When $k$ has characteristic zero and $q$ is a primitive $m^{th}$ root of unity {\em for m odd}, Jakobsen and Zhang found in \cite{J-Z} that \linebreak[4]${\rm PIdeg}\, \mathcal O_q(M_n(k)) = m^{\frac{n(n-1)}{2}}$ by using De Concini's and Procesi's tool given in Theorem \ref{dp}.  This result is reproved in \cite{J-J} using results of De Concini and Procesi and also J\o ndrup's work from \cite{jondrup}.  Now we can recover ${\rm PIdeg}\, \mathcal O_q(M_n(k)$ without the assumption that $k$ has characteristic zero.

In the single parameter case of $n \times n$ quantum matrices, the matrix that we use to calculate the PI degree is 

$$
\Lambda' = \begin{pmatrix} A_n &I_n &I_n &I_n &\cdots &I_n\\  -I_n &A_n &I_n &I_n &\cdots &I_n\\ -I_n&-I_n &A_n &I_n &\cdots &I_n\\ &&&\vdots\\ -I_n &-I_n &-I_n &-I_n &\cdots &A_n \end{pmatrix}
$$
where
$$
A_n= \begin{pmatrix} 0 &1 &1 &1 &\cdots &1\\ -1 &0 &1 &1 &\cdots &1\\ -1 &-1 &0 &1 &\cdots &1\\ &&&\vdots\\ -1 &-1 &-1 &\cdots &-1 &0 \end{pmatrix}
$$
is  $n \times n$ and $I_n$ is the $n\times n$ identity matrix.

For any $n$, the characteristic polynomial of $A_n$ is the sum of the terms of degree $\equiv n$ (mod 2) in the binomial expansion of $(x+1)^n$, so in fact $\chi_n(x) = \frac{1}{2} (x+1)^n + \frac{1}{2} (x-1)^n$.
 But there is also a recursion formula for the characteristic polynomial for $n \ge 3$ given by
\begin{equation*}
\chi_n(x) = \chi_{n-1}(x) (x+1) - (x-1)^{n-1},
\end{equation*}
which will be useful in the linear algebra that follows.

We will perform the following row reductions on the rows of blocks of $\Lambda'$.  For ease of notation, we'll denote the $j^{th}$ row of blocks as $BR_j$, the interchange of $BR_i$ and $BR_j$ as $BR_i\leftrightarrow BR_j$, and the addition of a multiple of $BR_i$ to $BR_j$ as $MBR_i + BR_j \mapsto BR_j$, where $M \in M_n(\mathbb Z)$.
\begin{itemize}
\item $BR_1\leftrightarrow BR_n$.
\item $-I_nBR_1 \mapsto BR_1$.
\item For $i = 2, \dotsc, n-1$, $BR_1 + BR_i \mapsto BR_i$.
\item $BR_n - A_nBR_1 \mapsto BR_n$. 
\end{itemize}
This yields the matrix

$$
\begin{pmatrix} I_n &I_n &I_n &I_n &\cdots &-A_n\\  0 &A_n + I_n &2I_n &2I_n &\cdots &I_n-A_n\\ 0&0 &A_n + I_n &2I_n &\cdots &I_n- A_n\\ \vdots&\vdots&&\ddots&&\vdots\\ 0&0&0&\cdots&A_n +I_n & I_n - A_n \\0 &I_n-A_n &I_n-A_n &I_n - A_n &\cdots  &I_n + A_n^2 \end{pmatrix}
$$
which can be reduced further by $n-2$ block row operations, each of which produces one zero block in the $n^{th}$ row.  We list the first three here along with the resulting $(n,n)$ block.
\begin{itemize}  
\item $(A_n + I_n)BR_n - (I_n - A_n)BR_2 \mapsto BR_n: \qquad A_n^3 + 3A_n$ 
\item $(A_n + I_n)BR_n + (I_n - A_n)^2 BR_3 \mapsto BR_n: \qquad A_n^4 + 6A_n^2 + I_n$
\item $(A_n + I_n)BR_n - (I_n - A_n)^3 BR_4 \mapsto BR_n: \qquad A_n^5 + 10A_n^3 + 5A_n$
\end{itemize}

In general, the block row operations that we need to perform in order to obtain a block upper triangular matrix are:
\begin{itemize}
\item For $i = 2, \dotsc, n-1$, $(A_n + I_n)BR_n + (-1)^{i-1}(I_n - A_n)^{i-1} BR_i \mapsto BR_n$.
\end{itemize}
These row operations are justified when $m$ is odd because $A_n + I_n$ is invertible in $M_n(\mathbb Z / m\mathbb Z)$ in that case, as will be shown below. 
After applying this step to the $i^{th}$ row, the $(n,n)$ block is $\chi_{i+1}(A_n)$.  So the resulting block upper triangular matrix is 

$$
\begin{pmatrix} I_n &I_n &I_n &I_n &\cdots &-A_n\\  0 &A_n + I_n &2I_n &2I_n &\cdots &I_n-A_n\\ 0&0 &A_n + I_n &2I_n &\cdots &I_n- A_n\\ \vdots&\vdots&&\ddots&&\vdots\\ 0&0&0&\cdots&A_n +I_n & I_n - A_n \\0 &0 &0 &0 &\cdots  &\chi_n(A_n) \end{pmatrix}
$$
where $\chi_n(A_n)$ is the $n \times n$ zero matrix.  Each block on the diagonal is 
$$
A_n+I_n = \begin{pmatrix} 1 &1 &1 &1 &\cdots &1\\ -1 &1 &1 &1 &\cdots &1\\ -1 &-1 &1 &1 &\cdots &1\\ &&&\vdots\\ -1 &-1 &-1 &-1&\cdots &1 \end{pmatrix}
$$
which can be row reduced just by adding row 1 to rows 2 through $n$ to yield the matrix
$$
\begin{pmatrix} 1 &1 &1 &1 &\cdots &1\\ 0 &2 &2 &2 &\cdots &2\\ 0 &0 &2 &2 &\cdots &2\\ &&&&\ddots\\ 0 &0 &0 &0&\cdots &2 \end{pmatrix}.
$$
In particular this shows that $A_n + I_n$ is invertible in $M_n(\mathbb Z / m\mathbb Z)$ for $m$ odd.
Hence $\Lambda'$ can be reduced through row operations to an upper triangular $n^2 \times n^2$ matrix with $2n-2$ ones, $(n-1)(n-2)$ twos, and $n$ zeroes on the diagonal.  Assuming that $q \in k$ is a primitive $m^{th}$ root of unity, and recalling Theorem \ref{dp}, the cardinality of the image in $(\mathbb Z / m \mathbb Z)^{n^2}$ is $m^{n^2 - n}$ if $m$ is odd.   Thus we conclude that ${\rm PIdeg}\,\mathcal O_q M_n(k) =  m^{\frac{n(n-1)}{2}}$, recovering the result of Jakobsen and Zhang \cite{J-Z} in characteristic zero.  By similar methods, one can show that  ${\rm PIdeg}\,\mathcal O_q M_n(k) =  m^{\frac{n(n-1)}{2}}/ 2^{\frac{(n-1)(n-2)}{2}}$ when $m$ is even.  For details on this result see \cite{J-Z} or \cite{haynal}.

\subsection{The algebra $K_{n, \Gamma}^{P, Q} (k)$, which generalizes the coordinate rings of even-dimensional quantum Euclidean space and quantum symplectic space}

For $P = (p_1, \dotsc, p_n)$ and $Q = (q_1, \dotsc, q_n)$ in $(k^{\times})^n$ with $p_i \ne q_i$ for all $i = 1, \dotsc , n$, and $\Gamma = (\gamma_{ij}) \in M_n(k^{\times})$ multiplicatively antisymmetric, the $k$-algebra $K_{n, \Gamma}^{P, Q} (k)$ introduced in \cite{horton} is defined by generators $x_i, y_i$ for $i = 1, \dotsc, n$ and relations

\begin{align*}  
y_i y_j & = \gamma_{ij} y_j y_i & \text{all } i, j\\
x_i x_j & = q_i p_j^{-1} \gamma_{ij} x_j x_i & i<j\\
x_i y_j & = p_j \gamma_{ji} y_j x_i  & i<j\\
x_i y_j & = q_j \gamma_{ji} y_j x_i & i>j\\
x_i y_i & = q_i y_i x_i + \sum_{\ell < i} (q_{\ell} - p_{\ell}) y_{\ell} x_{\ell} & \text{all } i.
\end{align*}

This algebra may be presented in the form of an iterated skew polynomial ring
\begin{equation*}
k[y_1][x_1; \tau_1][y_2; \sigma_2][x_2; \tau_2, \delta_2] \dotsb [y_n; \sigma_n][x_n; \tau_n, \delta_n]
\end{equation*}
where the automorphisms $\tau_i, \sigma_i$ and derivations $\delta_i$ are defined by
\begin{align*}
\sigma_i(y_j) & = \gamma_{ij} y_j & j<i\\
\sigma_i(x_j) & = p_i^{-1} \gamma_{ji} x_j & j<i\\
\tau_i (y_j) & = q_j \gamma_{ji} y_j & j<i\\
\tau_i (x_j) & = q_j^{-1} p_i \gamma_{ij} x_j & j<i\\
\tau_i (y_i) & = q_i y_i & \text{all } i\\
\delta_i (x_j) & = \delta_i (y_j) = 0 & j<i\\
\delta_i (y_i) & = \sum_{\ell < i} (q_{\ell} - p_{\ell}) y_{\ell} x_{\ell} & \text{all } i.
\end{align*}

Routine computations show that $\tau_i^{-1} \delta_i \tau_i (y_i) = q_i p_i^{-1} \delta_i (y_i)$ for all $i$, and so we conclude that each $(\tau_i, \delta_i)$ is a $q_i p_i^{-1}$-skew derivation.   For ease of notation we now shall let ${\bf k} = k[t_1^{\pm 1}, \dotsc, t_n^{\pm 1}, u_1^{\pm 1}, \dotsc, u_n^{\pm 1}]$ with $T = (t_1, \dotsc, t_n) \in {\bf k}$ and $U = (u_1, \dotsc, u_n) \in {\bf k}$.  We may present the ${\bf k}$-algebra $K_{n, \Gamma}^{T, U} ({\bf k})$ as an iterated skew polynomial ring   
\begin{equation*}
{\bf k}[y_1][x_1; \bar \tau_1][y_2; \bar \sigma_2][x_2; \bar \tau_2, \bar \delta_2] \dotsb [y_n;\bar \sigma_n][x_n; \bar \tau_n, \bar \delta_n]
\end{equation*}
where the automorphisms and derivations are defined analogously to those of $K_{n. \Gamma}^{P, Q} (k)$ with $t_i$ replacing $p_i$ and $u_i$ replacing $q_i$. Let $I \subseteq K_{n, \Gamma}^{T, U} ({\bf k})$ be the ideal generated by the $2n$ monomials $t_i - p_i, \, u_i - q_i$ for $i = 1, \dotsc, n$.  It is immediate that 
\begin{equation*}
K_{n, \Gamma}^{T, U} ({\bf k}) / I  \cong K_{n, \Gamma}^{P, Q} (k),
\end{equation*}
with each $\bar \tau_i, \, \bar \delta_i, \, \bar \sigma_i$ reducing to  $\tau_i, \, \delta_i, \, \sigma_i$ respectively. 

Let $A_j$ denote the subalgebra of $K_{n, \Gamma}^{T, U} ({\bf k})$ generated by $y_m, x_m$ for $m < j$ and $y_j$.  To show that $\bar \delta_j^i (A_j) \subseteq (i)!_{u_j t_j^{-1}} A_j$, it suffices to check that $\bar \delta_j^i (y_j)$ is an element of $(i)!_{u_j t_j^{-1}} A_j$ by Lemma \ref{containment}.  This is given for $i=1$ by the formula for $\bar \delta_j$ and is zero for $i>1$. 
So, by Theorem \ref{specialize}, each $\delta_i$ in our presentation of $K_{n, \Gamma}^{P, Q} (k)$ extends to an iterative, locally nilpotent h.$q_i p_i^{-1}$-s.$\tau_i$-d. on the appropriate subalgebra.  Then Corollary \ref{qaffine} gives
\begin{equation*}
{\rm Fract}\, K_{n, \Gamma}^{P, Q} (k) \cong {\rm Fract}\, \mathcal O_{\Lambda} (k^{2n})
\end{equation*}  
where the $2n \times 2n$ matrix of relations $\Lambda = (B_{ij})$ is comprised of $2 \times 2$ blocks
\begin{align*}
B_{ii} & = \begin{pmatrix} 1 &q_i^{-1}\\ q_i &1\end{pmatrix}, \text{ for all } i; \\
B_{ij} &= \begin{pmatrix} \gamma_{ij} &q_i^{-1} \gamma_{ji}\\ p_j \gamma_{ji} &q_i p_j^{-1} \gamma_{ij} \end{pmatrix}, \text{ for } i<j ;\\
B_{ij} &= \begin{pmatrix} \gamma_{ij} & p_i^{-1} \gamma_{ij}\\ q_j \gamma_{ji} &q_j^{-1}p_i \gamma_{ij} \end{pmatrix}, \text{ for } i>j.
\end{align*}

If the $q_i, p_i$ and $\gamma_i$ are all roots of unity, then $\mathcal O_{\Lambda} (k^{2n})$ is a PI algebra.  Suppose $q_i$ is an $r_i^{th}$ root of unity, $p_i$ is an $s_i^{th}$ root of unity, and $\gamma_{ij}$ is an $r_{ij}^{th}$ root of unity for all $i, j$.  Let $r = {\rm lcm}\{ r_i, s_i, \gamma_{ij} \mid i, j = 1, \dotsc, n\}$.  Then there extsis a primitive $r^{th}$ root of unity $q \in k$ and integers $b_i, c_i, b_{ij}$ such that $q_i = q^{b_i}$, $p_i = q^{c_i}$, and $\gamma_{ij} = q^{b_{ij}}$ for all $i, j$.  The powers of $q$ from the matrix $\Lambda$ provide the entries for an integer matrix $\Lambda'$ comprised of $2 \times 2$ blocks
\begin{align*}
B'_{ii} & = \begin{pmatrix} 0 &-b_i\\ b_i &0\end{pmatrix}, \text{ for all } i; \\
B'_{ij} &= \begin{pmatrix} b_{ij} &b_{ji} - b_i\\ b_{ji}+c_j &b_i + b_{ij} - c_j \end{pmatrix}, \text{ for } i<j ;\\
B'_{ij} &= \begin{pmatrix} b_{ij} & b_{ij} - c_i\\ b_{ji}+b_j &b_{ij} +c_i - b_j \end{pmatrix}, \text{ for } i>j.
\end{align*} 
Then ${\rm PIdeg}\, K_{n, \Gamma}^{P, Q} (k)$ can be calculated using $\Lambda'$ in Theorem \ref{dp} (2).

The coordinate ring of quantum Euclidean $2n$-space over $k$, $\mathcal O_q(\mathfrak{o}k^{2n})$, is formed by setting $q_i = 1, \, p_i = q^{-2}$ for all $i$, and $\gamma_{ij} = q^{-1}$ for $i<j$ in the parameters $Q,\,P$, and $\Gamma$ (see \cite {horton}, Example 2.6).  Then the integer matrix, $\Lambda'$, is
\begin{equation*}
\begin{pmatrix}
0 & 0 & -1 & 1 & -1 & 1 & \hdots & -1 & 1\\
0 & 0 & -1 & 1 & -1 & 1 & \hdots& -1 & 1\\
1 & 1 & 0 & 0 & -1 & 1 & & -1 & 1\\
-1 & -1 & 0 & 0 & -1 & 1 & & \vdots &\\
1 & 1 & 1 & 1 & 0 & 0 & & & \\
-1 & -1 & -1 & -1 & 0 & 0 & & & \\
\vdots & \vdots & \vdots  & \vdots & & & \ddots & & \\
1 & 1 & 1 & 1 & 1 & \hdots & & 0 & 0\\
-1 & -1 & -1 & -1 & -1 & \hdots & & 0 & 0
\end{pmatrix}.
\end{equation*}
We perform the following row reductions that preserve the size of the image of the \linebreak[4]homomorphism $\mathbb Z^{2n} \longrightarrow \mathbb Z^{2n}$ given by $\Lambda'$:
\begin{itemize}
\item For $j = 2n,\, 2n-1,\,2n-2,\dotsc,4$, replace row $j$ with $\text{row }j + \text{row }(j-1)$
\item Replace row 2 with $\text{row }2 - \text{row }1$
\item Replace the (new) row 5 with $\text{row } 5 + \text{row }1$
\item For $j = 4,\, 6,\, 8, \dotsc , 2n-4$, replace row $j$ with $\text{row }j + 2\text{row }(j+3)$
\item For $n \ge 4$, rearrange the rows to order $3, 1, 5, 7, 4, 9, 6, 11, \dotsc, 2i,\, 2i+5, \dotsc,\\ 2n-4,\, 2n-2,\, 2,\, 2n$.
\end{itemize}
The resulting matrix has the form
\begin{equation*}
\begin{pmatrix}
1 & 1 & & & & & & & & & \\
 & 0 & -1 & 1 & & & & & & &  \\
 &  & 0 & 2 & & & & &\boldsymbol{*} & & \\
 &  &  & 0 & 1 & 1 & & & & & \\
 &  &  &  & 0 & 4 & & & & & \\
 &  &  &  &  &  & \ddots &  & & & \\
 &  &  &  &  &  & 0 & 1 & 1 & &\\ 
 & & & {\bf 0} & & & & 0 & 4 & & \\
 &  &  &  &  & &  &  & 0 & -2 & 2\\
 &  &  & &  &  &  &  &  & 0 & 0\\
 &  &  & &  &  &  &  &  &  & 0
\end{pmatrix}.
\end{equation*}
When $n$ is even, the pivot in the third row does not divide all the entries in its row, so more elementary row and column operations are needed before it becomes clear that the matrix can be diagonalized.  By a method similar to that used in Example 5.1, suppressed here in the interest of saving space but listed explicitly in \cite{haynal}, we obtain the Smith normal form diag$(1,1,\dotsc,1,4,4,\dotsc,4,0,0)$ with $n$ ones and $n-2$ fours when $n$ is even; and diag$(1,1,\dotsc,1,2,2,4,4,\dotsc,4,0,0)$, with $n-1$ ones and $n-3$ fours when $n$ is odd.  Thus we have
\begin{equation}\label{evenqeuclidpi}
{\rm PIdeg}\, \mathcal O_q(\mathfrak{o}k^{2n}) = \begin{cases}
r^{n-1}, \qquad \qquad  &r \, \text{odd}\\
r^{n-1}/2^{\lfloor \frac{n-1}{2}\rfloor }, \quad &r \,\text{even} \notin 4\mathbb Z\\
r^{n-1}/2^{n-2}, \qquad  &r \in 4\mathbb Z
\end{cases}.
\end{equation}
The low-dimensional cases do not fit the same pattern, but the matrices for the cases $n=2$ and $n=3$ are readily transformed to 
\begin{equation*}
\begin{pmatrix}
1&1&0&0\\
0&0&-1&1\\
0&0&0&0\\
0&0&0&0
\end{pmatrix} \qquad \text{ and } \qquad
\begin{pmatrix}
1&1&0&0&-1&1\\
0&0&-1&1&-1&1\\
0&0&0&2&0&0\\
0&0&0&0&-2&2\\
0&0&0&0&0&0\\
0&0&0&0&0&0
\end{pmatrix}
\end{equation*}
respectively.    Therefore, formula (\ref{evenqeuclidpi}) holds for all $n \ge 2$.

As a specific case of $K_{n, \Gamma}^{P, Q} (k)$, quantum symplectic space $\mathcal O_q(\mathfrak{sp}(k^{2n}))$ is formed by setting $q_i=q^{-2}$ and $p_i=1$ for all $i$, and $\gamma_{ij}=q$ for $i<j$  (see \cite {horton}, Example 2.4).  With these parameters, the $2n \times 2n$ integer matrix $\Lambda'$ is 
\begin{equation*}
\begin{pmatrix}
0 & 2 & 1 & 1 & 1 & 1 & \hdots & 1 & 1\\
-2 & 0 & -1 & -1 & -1 & -1 & \hdots& -1 & -1\\
-1 & 1 & 0 & 2 & 1 & 1 & & 1 & 1\\
-1 & 1 & -2 & 0 & -1 & -1 & & -1 &-1\\
-1 & 1 & -1 & 1 & 0 & 2 & & \vdots & \\
-1 & 1 & -1 & 1 & -2 & 0 & &  & \\
\vdots & \vdots & \vdots  & \vdots & & & \ddots & & \\
-1 & 1 & -1 & 1 & -1 & 1 & \hdots & 0 & 2\\
-1 & 1 & -1 & 1 & -1 & 1 & \hdots & -2 & 0
\end{pmatrix}.
\end{equation*}
We perform the following row reductions that preserve the size of the image of the \linebreak[4]homomorphism $\mathbb Z^{2n} \longrightarrow \mathbb Z^{2n}$ given by $\Lambda'$:
\begin{itemize}
\item For $j = 2n, 2n-1,\dotsc, 4$, replace row j with $\text{row }j - \text{row }(j-1)$  
\item Replace row 2 with $-(\text{row }2 - 2\text{row }3 + \text{row }1)$
\item For $j=4,6,8,\dotsc,2n-2$, replace row j with $\text{row }j + 2\text{row }(j+1)$
\item For $n\ge3$, order the rows $3, 1, 5, 2, 7, 4, 9, \dots,2j, 2j+5, \dotsc, 2n-4, 2n, 2n-2$.
\end{itemize}
This yields a matrix whose image is more easily measured:
\begin{equation*}
\begin{pmatrix}
-1 & 1&&&&&&&&\\ 
0 & 2 &&&&&&&&\\
& 0 &1& 1&&&\boldsymbol{*}&& \\
&&0 &4& &&&&& \\
&&&& \ddots &&&&&\\
&&&&&1 & 1 &&&\\
&&\bf 0& &&0 & 4 &&&\\
&&&&&&&-2 & -2 \\ 
&&&&&&& 0 & 4 
\end{pmatrix}.
\end{equation*}
But the pivot in row 2 is problematic because it does not always divide the other entries in its row.  With further elementary row and column operations, full details of which can be found in \cite{haynal}, we can bring this matrix into Smith normal form diag$(1,1,\dotsc,1,4,4,\dotsc,4)$ with $n$ ones and $n$ fours when $n$ is even; or the form \linebreak[4]diag$(1,1,\dotsc,1,2,2,4,4,\dotsc,4)$ with $n-1$ ones, two twos, and $n-1$ fours when $n$ is odd.

For $n = 1, 2$, the row reduced matrices are, respectively, 
\begin{equation*}
\begin{pmatrix}
0& 2\\
-2 & 0
\end{pmatrix}
\qquad \text{ and } \qquad
\begin{pmatrix}
-1 & 1 & 0 & 2\\
0 & 1 & 1 & 1 \\
0 & 0 & -4 & -4\\
0 & 0 & 0 & -4
\end{pmatrix}.
\end{equation*}
Hence we have, for all $n$,
\begin{equation*}
{\rm PIdeg}\, \mathcal O_q(\mathfrak{sp}(k^{2n})) = \begin{cases}
r^n, \qquad \qquad  &r \text{ odd}\\
{r^n}/{2^{\lfloor \frac{n+1}{2} \rfloor  }},  \quad  &r \text{ even, }\, r\notin 4\mathbb Z\\
r^n/2^n, \qquad  \quad &r \in 4\mathbb Z
\end{cases}.
\end{equation*}

\section{Prime Factor Localizations}

In this section we present a structure theorem for completely prime factors of iterated skew polynomial rings analogous to the main theorem of section four.  Applying this result to the algebras studied in section five, we'd like to strengthen it to the form of the quantum Gel'fand-Kirillov conjecture.
Recall that the assumptions about skew polynomial rings from section one are still in effect.

\begin{theorem}\label{primefactor}
Let $A=R[x; \tau, \delta]$, where $R$ is noetherian and $\delta \tau = q \tau \delta$ for some $q \in k^{\times}$. Assume that $\delta$ extends to a locally nilpotent, iterative h.$q$-s.$\tau$-d., $\{ d_i \}$, on $R$.  Let $P \in \text{\rm{spec}}\, A$ be completely prime.  Then 

{\rm (1)} there exists a cyclic Ore set $S$ in $A/P$ such that $(A/P)S^{-1} \cong \big ( R[y; \tau]/Q \big ) Y^{-1}$ for some completely prime $Q \in \text{\rm{spec}}\, R[y; \tau]$ and cyclic Ore set $Y$,

{\rm (2)} ${\rm Fract}\,A/P \cong {\rm Fract}\,R[y;\tau]/Q$.
\end{theorem}

\begin{proof}
The completely prime ideal $P$ naturally satisfies one of two cases: $x \in P$ or $x \notin P$.  If $x \in P$, then $xA \subseteq P$ and $Ax \subseteq P$. So the relation $xr = \tau(r)x + \delta(r)$ implies that $\delta(r) \in P$ for all $r \in R$.  Hence, there is a completely prime ideal $I \in R$ such that $A/P \cong R/I \cong R[y; \tau]/(I+ \langle y \rangle )$.  In this case, we can take $S =  Y = \{ 1 \}$ and localize.  If $x \notin P$, then $x^i \notin P$ for all $i \in \mathbb N \cup \{ 0 \}$ because $A/P$ is a domain.  Letting $S = \{1, x, x^2, \dotsc \}$, which is a known denominator set in $A$, we have $P \cap S = \varnothing$.  Since extension and contraction provide inverse bijections between the sets $\text{spec}\, AS^{-1}$ and \linebreak[4]$\{I \in \text{spec}\, A \mid I \cap S = \varnothing \}$, we know that $P^e \in \text{spec}\, AS^{-1}$.  From Theorem \ref{summary}, we have $AS^{-1} \cong R[y^{\pm 1}; \tau]$, a localization of $R[y; \tau]$.  So there is a completely prime ideal $\bar Q \triangleleft R[y^{\pm 1}; \tau]$ such that $AS^{-1}/ P^e \cong R[y^{\pm 1}; \tau] / \bar Q$.  Setting $Y = \{1, y, y^2, \dotsc, \}$, contraction to $R[y; \tau]$ gives a completely prime ideal $Q$, where $Q \cap Y = \varnothing$, such that $R[y^{\pm 1}; \tau] / \bar Q$ is isomorphic to $(R[y; \tau] / Q)Y^{-1}$.  The canonical projection $\pi : AS^{-1} \longrightarrow (A/P)S^{-1}$ gives $AS^{-1}/P^e \cong (A/P)S^{-1}$.  Thus $(A/P)S^{-1} \cong \big( R[y; \tau] / Q \big) Y^{-1}$.
\end{proof}

\begin{theorem}
Let $R$ be a noetherian $k$-algebra, and let 
\begin{equation*}
A= R[x_1, \tau_1, \delta_1] \dotsb [x_n; \tau_n, \delta_n]
\end{equation*}
be an iterated skew polynomial ring where, for $j<i$ and $\lambda_{ij} \in k^{\times}$,  $\tau_i(x_j) = \lambda_{ij}x_j$, and $\delta_i$ is a $q_i$-skew $\tau_i$-derivation, $q_i \ne 1$, which extends to a locally nilpotent, iterative h.$q_i$-s.$\tau_i$-d. $\{d_{i,p} \}_{p=0}^{\infty}$ on $R[x_1; \tau_1, \delta_1] \dotsb [x_{i-1}; \tau_{i-1}, \delta_{i-1}]$ for all $i$.  Let $A' = R[y_1; \tau_1'][y_2; \tau_2'] \dotsb [y_n; \tau_n']$ where $\tau'_i (y_j) = \lambda_{ij} y_j$ for all $i$ with $j<i$ and the same units $\lambda_{ij}$ as above.  Let $P$ be a completely prime ideal in $A$.  Then 

{\rm (1)} there exists a finitely generated Ore set $S_n$ in $A/P$ such that $(A/P)S_n^{-1}$ is isomorphic to $\big ( A'/ Q \big ) Y_n^{-1}$ for some completely prime ideal $Q \subseteq A'$ and finitely generated Ore set $Y_n$,

{\rm (2)}  ${\rm Fract}\,A/P \cong {\rm Fract}\,A'/Q$.
\end{theorem}

\begin{proof}
The case $n=1$ has been established in Theorem \ref{primefactor}.  Suppose the result holds for the case $n-1$, and let $A_{n-1} = R[x_1, \tau_1, \delta_1] \dotsb [x_{n-1}; \tau_{n-1}, \delta_{n-1}] \subseteq A$.  Then we have $A = A_{n-1}[x_n; \tau_n ,\delta_n]$.    If $x_n \in P$, then as in Theorem \ref{primefactor} there is a completely prime ideal $I \subseteq A_{n-1}$ such that $A/P \cong A_{n-1}/I \cong A_{n-1} [y_n; \tau'_n] /(I + \langle y_n \rangle )$.  The induction hypothesis and Lemma \ref{switchingB} imply that $\big ( A_{n-1} [y_n; \tau'_n] /(I + \langle y_n \rangle ) \big )S^{-1} \cong \big ( A'/Q \big ) Y^{-1}$ for some finitely generated Ore sets $S$ and $Y$.  Hence there is a finitely generated Ore set $S_n$ in $A$ such that $(A/P)S_n^{-1} \cong (A'/Q)Y^{-1}$.  If $x_n \notin P$, let $S_n = \{1,\, x_n,\, x_n^2, \dotsc\,\} \subseteq A$ and $Y_n = \{1,\, y_n,\, y_n^2, \dotsc\, \} \subseteq A_{n-1}[y_n; \tau_n]$.  Then from the single-variable result, it follows that 
\begin{equation}\label{primefactloc}
\big ( A/P \big ) S_n^{-1} \cong \big ( A_{n-1}[y_n; \tau_n'] / \bar Q \big ) Y_n^{-1},
\end{equation}
for a completely prime ideal $\bar Q \subseteq A_{n-1}[y_n; \tau_n']$.  From Lemma \ref{switchingB}, we have 
\begin{equation*}
A_{n-1}[y_n; \tau_n'] = R[y_n; \tau_n'][x_1; \tau'_1, \delta'_1] \dotsb [x_{n-1}; \tau'_{n-1}, \delta'_{n-1}],
\end{equation*}
which is an iterated skew polynomial ring in $n-1$ variables over the coefficient ring $R[y_n; \tau_n']$ that satisfies the current assumptions.  So, we apply the induction hypothesis and rearrange variables to obtain 
\begin{align*}
\big ( A_{n-1}[y_n; \tau_n] / \bar Q \big ) Y_n^{-1} & \cong \big ( R[y_n; \tau_n'][y_1; \tau_1'] \dotsb [y_{n-1}; \tau_{n-1}'] / Q \big ) Z^{-1}\\
& \cong \big ( R[y_1; \tau_1'][y_2; \tau_2'] \dotsb [y_n; \tau_n'] / Q \big ) Z^{-1}
\end{align*}
for a completely prime ideal $Q \subseteq R[y_1; \tau_1'][y_1; \tau_1'] \dotsb [y_n; \tau_n']$ and a  denominator set\linebreak[4] $Z \subseteq R[y_1; \tau_1'][y_1; \tau_1'] \dotsb [y_n; \tau_n'] / Q$. This, along with isomorphism  (\ref{primefactloc}) gives the result.
\end{proof}

When $R$ is replaced by $k$, we have the following result.

\begin{corollary} 
Let $A = k[x_1, \tau_1, \delta_1] \dotsb [x_n; \tau_n, \delta_n]$, where $\tau_i(x_j) = \lambda_{ij}x_j$ and $\delta_i \tau_i =~q_i \tau_i \delta_i$, $q_i \ne 1$, for $\lambda_{ij}, \, q_i \in k^{\times}$ and all $i$ with $j < i$. Assume that each $\delta_i$ extends to a locally nilpotent, iterative h.$q_i$-s.$\tau_i$-d. $\{d_{i,m} \}_{m=0}^{\infty}$ on the subalgebra $k[x_1; \tau_1, \delta_1] \dotsb [x_{i-1}; \tau_{i-1}, \delta_{i-1}]$.  Let $P$ be a completely prime ideal in $A$ and set $\lambda_{ii} = 1$ and $\lambda_{ji} = \lambda_{ij}^{-1}$. Then for \linebreak[4]$\boldsymbol \lambda = (\lambda_{ij}) \in M_n(k)$, and an appropriate completely prime ideal $Q \subseteq \mathcal O_{\boldsymbol \lambda} (k^n)$, we have 
\begin{equation*}
\text{\rm Fract}\,A/P \cong \text{\rm Fract}\, \mathcal O_{\boldsymbol \lambda} (k^n) / Q.
\end{equation*}
\end{corollary}

We summarize how this applies to the $k$-algebras of quantized coordinate type.

\begin{corollary}
Let $A$ be any of the examples discussed in sections {\rm 5.1 - 5.4}, and let $P$ be a completely prime ideal of $A$. Then there exist a positive integer $N$, a multiplicatively antisymmetric $N \times N$ matrix $\boldsymbol \lambda$ over $k$, and a completely prime ideal $Q \in \mathcal O_{\boldsymbol \lambda} (k^N)$ such that ${\rm Fract}\, A/P \cong {\rm Fract}\, \mathcal O_{\boldsymbol \lambda} (k^N)/Q$.
\end{corollary}

To complete the question posed by the corollary, one might ask how far the quantum Gel'fand-Kirillov conjecture extends to prime factor algebras.  For instance:

\begin{question}\label{openques}
Find conditions under which we can conclude that for any positive integer $n$, multiplicatively antisymmetric matrix $\boldsymbol \lambda \in M_n(k^{\times})$, and completely prime ideal \linebreak[4]$Q \in {\rm spec}\, \mathcal O_{\boldsymbol \lambda} (k^n)$, we have
\begin{equation*}
{\rm Fract}\, \mathcal O_{\boldsymbol \lambda} (k^n)/Q \cong {\rm Fract}\, \mathcal O_{\boldsymbol p} (K^m)
\end{equation*}
for some field extension $K\supseteq k$, integer $m \le n$, and $m \times m$ matrix $\boldsymbol p$ over $K$.
\end{question}

The case $n=1$ is trivial.  When $n=2$ and $Q$ contains $x_1$ or $x_2$, then ${\rm Fract}\, \mathcal O_{\boldsymbol \lambda}(k^2)/Q$ is isomorphic either to ${\rm Fract}\, \mathcal O_{\boldsymbol p}(k(y))$ where $\boldsymbol p = (1)$, or to $k$ itself.  In fact, for any $n$, if $Q$ is generated by a subset $S$ of $\{x_1, \dotsc, x_n\}$, then the result holds, with $\boldsymbol p$ the submatrix of $\boldsymbol \lambda$ formed by deleting the $i^{th}$ row and column for $x_i \in S$, and $K = k$.  When $x_i \notin Q$ for all $i$, answering the question fully will likely require different methods depending on the presence of roots of unity among the $\lambda_{ij}$.  A positive answer in the generic case has been provided in the proof of \cite[Theorem~2.1]{primefac}:

\begin{theorem}
{\rm [Goodearl - Letzter]} Let k be a field, $\boldsymbol \lambda = (\lambda_{ij})$ a multiplicatively antisymmetric $n \times n$ matrix over $k^{\times}$, and $\boldsymbol \Lambda$ the subgroup of $k^{\times}$ generated by the $\lambda_{ij}$.  If $\boldsymbol \Lambda$ is torsionfree, then all of the prime ideals $Q$ of $\mathcal O_{\boldsymbol \lambda}(k^n)$ are completely prime.
\end{theorem}

In their proof, they showed that ${\rm Fract}\, \mathcal O_{\boldsymbol \lambda} (k^n)/Q \cong {\rm Fract}\, \mathcal O_{\boldsymbol p} (K^m)$, and identified $K$ as the quotient field of a commutative domain embedded in the center of $\mathcal O_{\boldsymbol \lambda} ((k^{\times})^n)/Q'$, where $Q'$ is the prime ideal in $\mathcal O_{\boldsymbol \lambda} ((k^{\times})^n)$ induced by localization.

Quantum affine space is included in a class called {\em quantum solvable algebras} by A.~N.~Panov.  The main theorem of \cite[Section 3]{panov2}, states that when the group generated by the $\lambda_{ij}$ is torsionfree, then ${\rm Fract}\, \mathcal O_{\boldsymbol \lambda} (k^n)/Q $ is isomorphic to the quotient division ring of a quantum torus.  The main theorem of \cite[Section 3]{panov3}, allows roots of unity and states that when $Q$ satisfies the extra condition of being stable under a certain set of derivations, then ${\rm Fract}\, \mathcal O_{\boldsymbol \lambda} (k^n)/Q $ is isomorphic to the quotient division ring of a quantum torus.  Cauchon's work may also be specialized to apply to quantum affine space when the group generated by the $\lambda_{ij}$ is torsionfree.  The result of \cite[Theorem~6.1.1]{cau}, indicates that ${\rm Fract}\, \mathcal O_{\boldsymbol \lambda} (k^n)/Q $ is isomorphic to ${\rm Fract}\, \mathcal O_{\boldsymbol p} (K^m)$ which specializes to this result.  But the division ring of real quaternions provides an example showing that Question \ref{openques} needs to have some conditions imposed.  Note that
\begin{equation*}
\mathbb H \cong \mathcal O_{\boldsymbol \lambda} (\mathbb R^3)/Q, \: \text{where} \:
\boldsymbol \lambda = \begin{pmatrix}
1 & -1 & -1\\
-1 & 1 & -1\\
-1 & -1 & 1
\end{pmatrix}, \: \text{and} \: Q = \langle x_1^2 +1, \, x_2^2 +1,\,  x_3^2 +1 \rangle.
\end{equation*}
Therefore, we cannot obtain the desired isomorphism of quotient division rings in this case, illustrating the necessity of an extra condition such as the one imposed by Panov in \cite{panov3}.

\section*{acknowledgments}
The author thanks her dissertation advisor, Ken Goodearl, for his direction that was so freely given in many inspiring discussions.


\begin{thebibliography}{30}

\bibitem{Alev} \textsc{J. Alev and F. Dumas}, {\em Sur le corps de fractions de certaines alg$\grave{e}$bres quantiques}, J. Algebra {\bf 170} (1994), 229-265

\bibitem{ast} \textsc{M. Artin, W. Schelter, and J. Tate}, {\em Quantum deformations of $GL_n$}, Comm. Pure Appl. Math {\bf 44} (1991), 879-895

\bibitem{Bourbaki} {\sc N. Bourbaki}, {\em $\acute E$l$\acute e$ments de math$\acute e$matique, Livre II, Alg$\grave e$bre, Chapitre 9, Formes sesquilin$\acute e$aires et formes quadratiques}, Hermann, Paris, 1959

\bibitem{barcelona} {\sc K. A. Brown and K. R. Goodearl}, {\em Lectures on Algebraic Quantum Groups}, Birkh\"auser Verlag, Basel - Boston, 2002

\bibitem{cau} {\sc G. Cauchon}, {\em Effacement des d$\acute{e}$rivations et spectres premiers des alg$\grave{e}$bres quantiques}, J. Algebra {\bf 260} (2003), 476-518

\bibitem{cau2} {\sc G. Cauchon}, {\em Spectre premier de $\mathcal O_q(M_n(k))$ image canonique et s$\acute {e}$paration normale}, J. Algebra {\bf 260} (2003), 519-569

\bibitem{cliff} {\sc G. Cliff}, {\em The division ring of quotients of the coordinate ring of the quantum general linear group}, J. London Math. Soc. (2) {\bf 51} (1995), 503-513

\bibitem{DP} {\sc C. De Concini and C. Procesi}, {\em Quantum Groups}, in D-modules Representation Theory, and Quantum Groups (Venezia, June 1992) (G. Zampieri and A. D'Agnolo, eds.), Lecture Notes in Math. 1565, Springer-Verlag, Berlin, 1993, 31-140

\bibitem{uniranks} {\sc K. R. Goodearl}, {\em Uniform ranks of prime factors of skew polynomial rings}, in Ring Theory, Proc. Biennial Ohio State - Denison Conf., 1992 (S. K. Jain and S. T. Rizvi, eds.), World Scientific, Singapore, 1993, 182-199

\bibitem{primesprqwa} {\sc K. R. Goodearl}, {\em Prime ideals in Skew polynomial rings and quantized Weyl algebras}, Trans. Amer. Math. Soc. {\bf 352} (2000), 1381-1403

\bibitem{primespec} {\sc K. R. Goodearl}, {\em Prime spectra of quantized coordinate rings}, in Interactions between Ring Theory and Representations of Algebras (Murcia 1998) (F. Van Oystaeyen and M. Saor\'in, eds.), Dekker, New York, 2000, pp. 205-237

\bibitem{catenarity} {\sc K. R. Goodearl and T. H. Lenagan}, {\em Catenarity in quantum algebras}, J. Pure and Appl. Algebra {\bf 111} (1996), 123-142

\bibitem{primefac} {\sc K. R. Goodearl and E. S. Letzter}, {\em Prime factor algebras of the coordinate ring of quantum matrices}, Proc. Amer. Math. Soc. {\bf 121} (1994), 1017-1025

\bibitem{memoirs} {\sc K. R. Goodearl and E. S. Letzter}, {\em Prime ideals in skew and q-skew polynomial rings}, Mem. Amer. Math. Soc. {\bf 521} (1994)

\bibitem{DMequiv} {\sc K. R. Goodearl and E. S. Letzter}, {\em The Dixmier-Moeglin equivalence in quantum coordinate rings and quantized Weyl Algebras}, Trans. Amer. Math. Soc. {\bf 352} (2000), 1381-1403

\bibitem{bluebook} {\sc K. R. Goodearl and R. B. Warfield, Jr.}, {\em An Introduction to Noncommutative Noetherian Rings, 2nd ed.}, Cambridge Univ. Press, Cambridge, 2004

\bibitem{haynal} {\sc H. A. Haynal}, {\em Pi degree parity in $q$-skew polynomial rings}, Ph.D. Thesis, to appear, (2007) University of California, Santa Barbara

\bibitem{horton} {\sc K. L. Horton}, {\em The prime and primitive spectra of multiparameter quantum symplectic and euclidean spaces}, Comm. Algebra {\bf 31} (10) (2003), 4713-4743

\bibitem{J-J} {H. P. Jakobsen and S. J\o ndrup}, {\em Quantized rank r matrices}, J. Algebra {\bf 246} (2001), 70-96,   arXiv:math.QA/9902133 v3, 23 May 2001

\bibitem{J-Z} {\sc H. P. Jakobsen and H. Zhang}, {\em The center of the quantized matrix algebra}, J. Albegra {\bf 196} (1997), 458-474

\bibitem{jondrup} {\sc S. J\o ndrup}, {\em Representations of skew polynomial algebras}, Proc. Amer. Math Soc. {\bf 128} (2000), 1301-1305

\bibitem{jondrup2} {\sc S. J\o ndrup}, {\em Representations of some PI algebras}, Comm. Algebra {\bf 31} (6) (2003), 2587-2602

\bibitem{jordan} {\sc D. A. Jordan}, {\em A simple localization of the quantized Weyl algebra}, J. Algebra {\bf 174} (1995), 267-281

\bibitem{lam} {\sc T. Y. Lam}, {\em Lectures on Modules and Rings}, Springer, New York, 1999

\bibitem{malm} {\sc D. R. Malm}, {\em Simplicity of partial and Schmidt differential operator rings}, Pacific J. Math. {\bf 132} (1998), no. 1, 85-112 

\bibitem{malt} {\sc G. Maltsiniotis}, {\em Calcul diff$\acute{e}$rentiel quantique}, Groupe de travail, Universit$\acute{e}$ Paris VII (1992)

\bibitem{brownbook} {\sc J. C. McConnell and J. C. Robson}, {\em Noncommutative Noetherian Rings}, Wiley-Interscience, Chichester - New York, 1987

\bibitem{mospan} {\sc V. G. Mosin and A. N. Panov}, {\em Division rings of quotients and central elements of multiparameter quantizations}, Sbornik: Mathematics {\bf 187}:6 (1996), 835-855

\bibitem{musson} {\sc I. M. Musson}, {\em Ring theoretic properties of the coordinate rings of quantum symplectic and Euclidean space}, in Ring Theory, Proc. Biennial Ohio State-Denison Conf., 1992 (S.K. Jain and S.T. Rizvi, eds.), World Scientific, Singapore, 1993, 248-258

\bibitem{Newman} {\sc M. Newman}, {Integral Matrices}, Academic Press, 1972

\bibitem{oh} {\sc S. Q. Oh}, {\em Catenarity in a class of iterated skew polynomial rings}, Comm. Algebra {\bf 25} (1) (1997), 37-49

\bibitem{panskew} {\sc A. N. Panov}, {\em Skew fields of twisted rational functions and the skew field of rational functions on $GL_q(n,K)$}, St. Petersburg Math J. {\bf 7} (1) (1996), 129-143 

\bibitem{panov1} {\sc A. Panov}, {\em Fields of fractions of quantum solvable algebras}, J. Algebra {\bf 236} (2001), 110-121

\bibitem{panov2} {\sc A. Panov}, {\em Stratification of prime spectrum of quantum solvable algebras}, Comm. Algebra {\bf 29}(9) (2001), 3801-3827

\bibitem{panov3} {\sc A. Panov}, {\em Quantum solvable algebras. Ideals and representations at roots of 1}, Transformation Groups {\bf 7}, no. 4, (2002) 379-402

\bibitem{RTF} {\sc N. Yu. Reshetikhin, L. A. Takhtadzhyan, and L. D. Fadeev}, {\em Quantization of Lie Groups and Lie Algebras}, Leningrad Math J. {\bf 1} (1990), 193-225

\bibitem{rowen} {\sc L. H. Rowen}, {\em Ring Theory}, Volumes I and II, Academic Press, Boston, 1988

\bibitem{smith} {\sc S. P. Smith}, {\em Quantum groups: An introduction and survey for ring theorists}, in Noncommutative Rings (S. Montgomery and L. W. Small, eds.), pp131-178, MSRI Publ. 24, Springer-Verlag, Berlin (1992)

\bibitem{stanley} {\sc R. P. Stanley}, {\em Enumerative Combinatorics}, Vol. I, Wadsworth \& Brooks/Cole, Monterey, CA, 1986

\end{thebibliography}
\end{document}